\documentclass[11pt,a4paper,reqno]{amsart}

\usepackage{amsmath}
\usepackage{amssymb}

\newtheorem{thm}{Theorem}[section]
\newtheorem{lem}[thm]{Lemma}
\newtheorem{prop}[thm]{Proposition}
\newtheorem{cor}[thm]{Corollary}
\newtheorem{defn}[thm]{Definition}

\theoremstyle{remark}
\newtheorem{remark}[thm]{Remark}
\newtheorem{exam}[thm]{Example}

\newcommand{\map}[4]{#1:#2^{#3}\rightarrow#4}
\newcommand{\inner}[2]{\langle#1,#2\rangle}
\newcommand{\N}{\mathbb{N}} 
\newcommand{\Z}{\mathbb{Z}} 
\newcommand{\C}{\mathbb{C}} 

\newcommand{\R}{\mathbb{R}}

\newcommand{\V}{\mathbb{V}}

\renewcommand{\S}{\mathbb{S}}
\newcommand{\RP}{\mathbb{RP}} 

\usepackage{hyperref}

\title[Equivariant constructions of Zoll families]{Equivariant constructions of spheres with Zoll families of minimal spheres}

\author{Lucas Ambrozio and Diego Guajardo}

\address{L. Ambrozio: IMPA \\ Rio e Janeiro RJ 22460-320
Brazil}
\email{l.ambrozio@impa.br}

\address{D. Guajardo: ICMC - USP \\ São Carlos SP 13566-590 Brazil}
\email{diegonguajardo@usp.br}

\thanks{L.A. is supported  by CNPq (309908/2021-3 - Bolsa PQ and 406666/2023-7 - Universal) and by FAPERJ  (grant SEI-260003/000534/2023 - BOLSA E-26/200.175/2023 and grant SEI-260003/001527/2023 - APQ1 E-26/210.319/2023). D. G. was financed, in part, by the São Paulo Research Foundation (FAPESP), Brasil. Process Number 2024/03493-2.}

\begin{document}

\begin{abstract}
	We construct one-parameter deformations of the Euclidean sphere $\mathbb{S}^n$ inside $\mathbb{R}^{n+1}$ that admit a Zoll family of codimension one embedded minimal spheres, in all dimensions $n\geq 2$. For $n=2$, this means that all geodesics are closed, embedded, and of the same period.  The method of construction is equivariant with respect to the natural actions of the orthogonal group. In particular, we show that the original Zoll spheres of revolution in $\mathbb{R}^3$ have counterparts in the context of minimal surface theory, in all dimensions. 
	
	We also describe the first examples of metrics on the real projective spaces $\mathbb{RP}^n$, in all dimensions $n \geq 3$, that admit a Zoll family of embedded minimal projective hyperplanes, and which are not isometric to metrics with minimal linear projective hyperplanes.
	
	The new constructions are underpinned by equivariant versions of Nash-Moser-Hamilton implicit function theorem, and yield new information even in dimension $n=2$. As an application, we also show that every finite group of the orthogonal group $O(3)$ that does not contain $-Id$ is the isometry group of some (classical) Zoll metric on $\mathbb{S}^2$.
\end{abstract}

\maketitle

\section{Introduction}
    Let $(M,g)$ be a connected Riemannian manifold. The metric $g$ is \textit{Zoll} if all geodesics are closed, simple, and have the same length. 
	
	The canonical metrics on spheres and projective spaces (\textit{i.e.} the standard metrics of the compact rank-one symmetric spaces) are Zoll metrics. All manifolds admitting Zoll metrics share topological properties with these manifolds (see \cite{Bes}, Chapter 7). It is expected that, on the projective spaces, the symmetric metrics are in fact the only Zoll metrics. This conjecture has been confirmed in the case of the real projective space $\mathbb{RP}^n$ in dimension $n=2$ by L. Green \cite{Gre} and in all other dimensions by M. Berger (see \cite{Bes}, Corollary D.2 together with Theorem 7.23).
	
	Perhaps surprisingly, spheres of any dimension have many Zoll metrics, aside from constant curvature metrics. The first examples were found by O. Zoll at the beginning of the twentieth century \cite{Zoll}. His examples are, in fact, spheres of revolution in $\mathbb{R}^3$. The construction was later extended, and a simple formula, depending only on an odd smooth function of one real variable, describes all Zoll metrics on the two-dimensional sphere that admit an isometric circle action (see \cite{Bes}, Corollary 4.16). Finally, Weinstein observed that the spherical suspension of these two-dimensional Zoll metrics yields Zoll metrics on spheres, of arbitrary dimension $n\geq 3$, which are invariant by the special orthogonal group $SO(n)$ (see \cite{Bes}, Proposition 4.41).   
	
	A different strategy for the construction of Zoll metrics on spheres, not based on symmetry reduction, was pursued by P. Funk \cite{Funk}, and later completed by V. Guillemin \cite{Guillemin}. Let $\mathbb{S}^n$ denote the Euclidean sphere of radius one centered at the origin of $\mathbb{R}^{n+1}$, and denote by $can$ its canonical metric. Funk deduced a necessary condition for the existence of a one-parameter family of Zoll metrics $g_t=e^{\rho_t}can$ on $\mathbb{S}^2$, with geodesics of length $2\pi$, starting at $g_0=can$. Namely, he showed that the time zero derivative $\dot{\rho}_0$ must be an odd function on $\mathbb{S}^2$ (with respect to the antipodal map). Then, Guillemin showed how to construct, out of any given smooth odd function $f$ on $\mathbb{S}^2$, a smooth one-parameter family of Zoll metrics $g_t=e^{\rho_t}can$ such that $g_t=(1+tf + o(t))can$ as $t$ goes to zero. In contrast with the two dimensional case, very little is known about the space of Zoll deformations in higher dimensions, since there are higher order obstructions for a Zoll deformation starting at $can$ to be possible (\cite{Kiyohara}).
	
	Due to the special yet flexible geometric properties of Zoll metrics on spheres, its study has attracted attention over the years from many different perspectives. While the classification of Zoll metrics remains incomplete (\textit{cf}. the comprehensive and still fairly updated classical book by A. Besse \cite{Bes}), some interesting recent results have highlighted the importance of Zoll metrics from variational points of view. See, for instance, \cite{ABS} (Corollary 4 and the subsequent discussion), \cite{AlvBal} (Theorem 1.6), and \cite{MS}. \\

	Geodesics have natural higher-dimensional generalizations: minimal submanifolds. While only geodesics can be studied with tools of the theory of dynamical systems, both geodesics and minimal submanifolds are variational objects, in the sense that they are critical points of the length and area functional, respectively.
	
	 There are several reasons to seek generalizations of Zoll metrics in the context of minimal submanifold theory, in particular because they may be special with respect to certain variational problems in the space of Riemannian metrics (\cite{AmbMon} and \cite{Luzzi}), and are relevant to min-max theory \cite{ACM2}. The next general definition seems to be appropriate to serve as a basis for fruitful investigations. It is modeled on the notion of a transitive family in a three-dimensional manifold introduced in \cite{GalMir}, and with the particular case of spheres introduced in \cite{ACM}.
	
	Let $(M,g)$ be a compact connected Riemannian manifold. A \textit{Zoll family of minimal hypersurfaces} in $(M,g)$ is a family of embedded compact connected minimal hypersurfaces $\{\Sigma_\sigma\}$, smoothly parametrized by points $\sigma$ of a compact connected smooth manifold, that satisfies the following two conditions: 
	\begin{itemize}
		\item[$a)$] For every point $(p,\pi)\in Gr_{n-1}(M)$ in the Grassmanian of (unoriented) tangent hyperplanes, there exists a unique $\Sigma_\sigma$ in the family such that $T_p\Sigma_\sigma=\pi$; and
		\item[$b)$] The assignment thus defined $(p,\pi)\in Gr_{n-1}(M)\mapsto \Sigma_\sigma$ is smooth.
	\end{itemize}

	The axioms above capture key properties of the set of geodesics of a Zoll metric on $\mathbb{S}^2$. In fact, the set of geodesics (modulo reparametrization) of a Zoll metric on $\mathbb{S}^2$ is known to be a set of embedded geodesic circles, smoothly parametrized by $\mathbb{RP}^2$ (see \cite{LeBMas}, Proposition 2.11), and whose elements depend smoothly on the point and tangent lines that define them uniquely. Notice also that the smoothness and connectedness of the family implies that all elements of the Zoll family of minimal hypersurfaces have the same area, because they are all critical points of the area functional.
	
	Denote by $\mathcal{Z}$ the set of Riemannian metrics on a given compact connected manifold that admit a Zoll family of minimal hypersurfaces. The set $\mathcal{Z}$ is preserved by dilation and pull-back by diffeomorphisms of the manifold. 
	
	For instance, on $\mathbb{S}^n$ (resp. $\mathbb{RP}^n$), the canonical metric $can$  belongs to $\mathcal{Z}$, because the set of codimension one totally geodesic spheres (resp. projective hyperplanes) is a Zoll family of minimal hypersurfaces in $(\mathbb{S}^n,can)$ (resp. $(\mathbb{RP}^n,can)$). The classification of minimal spheres in a homogeneous $(\mathbb{S}^3,g)$ implies that they form a Zoll family of minimal surfaces (see \cite{MeeksMiraPerezRos}).
	
	In \cite{ACM}, F. Marques, A. Neves and the first author constructed new metrics on $\mathbb{S}^n$ that belong to $\mathcal{Z}$, in all dimensions $n\geq 3$, using perturbation methods. For instance, they showed that for every smooth odd function $f$ on $\mathbb{S}^n$, $n\geq 3$, there exists a smooth one-parameter family of metrics $g_t=e^{\rho_t}can \in \mathcal{Z}$ such that $\rho_t = tf+o(t)$ as $t$ goes to zero (see \cite{ACM}, Theorem A). This is the exact counterpart of Guillemin's theorem in the context of minimal hypersurface theory. 
	
	The constructions in \cite{ACM} and \cite{Guillemin} yield only a partial understanding of the geometry of the metrics $g_t\in \mathcal{Z}$. In particular, it is not clear whether one can construct the family $g_t\in \mathcal{Z}$ in such a way that $g_t$ are the induced metrics on graphical perturbations of $\mathbb{S}^n$ inside $\mathbb{R}^{n+1}$, or so that $(\mathbb{S}^n,g_t)$ has a prescribed set of symmetries (for the case of the trivial isometry group, see, however, \cite{Bes}, Corollary 4.71 and \cite{ACM}, Theorem C).  \\
	
	We aim to fill some gaps in the understanding of the possible geometries of metrics in $\mathcal{Z}$, and to construct new examples with interesting geometric features.
	
	Our first result shows that there are non-trivial smooth one-parameter families of deformations of $\mathbb{S}^n$ inside $\mathbb{R}^{n+1}$ such that their induced metrics lie in $\mathcal{Z}$. More precisely, we have:

\begin{thm}\label{thmA}
    Let $f$ be a smooth odd function on $\mathbb{S}^n$, $n\geq 2$. Then there exists a smooth one-parameter family of functions $\psi_t\in C^{\infty}(\mathbb{S}^n)$, $|t|<\varepsilon$, with $\psi_t = tf + o(t)$ as $t$ goes to zero, such that  
    \begin{align}\label{eqisom}
        \iota_t: & \,\mathbb{S}^n\rightarrow\R^{n+1} \nonumber \\
        & x\rightarrow e^{\psi_t(x)}x,
    \end{align}
   is an embedding with $g_{t}=\iota_t^*can\in\mathcal{Z}$ for all $t\in (-\varepsilon,\varepsilon)$.
\end{thm}

	The Zoll families of minimal hypersurfaces on $(\mathbb{S}^n,g_t)$ constructed in Theorem \ref{thmA} are parametrized by $\mathbb{RP}^n$ and consist of the image by $\iota_t$ of graphical perturbations of the codimension one totally geodesic spheres of $(\mathbb{S}^n,can)$. Also, being odd is a necessary condition on the first-order term of the expansion of $\psi_t$ if the induced metrics $g_t=i^*_tcan$ belong to $\mathcal{Z}$ and have all elements of the Zoll family with area $area(\mathbb{S}^{n-1},can)$, see Remark \ref{rmkallvariations}.
	 
	It is important to bear in mind that the translations of $\mathbb{S}^n$ in the direction $v\in \mathbb{R}^{n+1}$ produce a family of isometric embeddings as in \eqref{eqisom} such that $\dot{\psi}_0(x)=\langle x,v\rangle$ for all $x\in \mathbb{S}$. Thus, any smooth odd function on $\mathbb{S}^n$ that is $L^2$-orthogonal to all (restrictions of) linear functions does in fact generate, via Theorem \ref{thmA}, a non-trivial deformation of $\mathbb{S}^n$ inside $\mathbb{R}^{n+1}$ with induced metric in $\mathcal{Z}$ (a similar remark can be made in the context of conformal deformations considered in \cite{ACM}, because one-parameter families of conformal diffeomorphisms of $(\mathbb{S}^n,can)$ generate families $g_t=e^{\rho_t}can \in \mathcal{Z}$ with $\dot{\rho}_0=\langle x,v\rangle$ for some vector $v\in \mathbb{R}^{n+1}$).
		
	The specific construction of the family $g_t$ in Theorem \ref{thmA}, out of a given smooth odd function $f$ on $\mathbb{S}^n$, turns out to be equivariant under the natural action of the orthogonal group $O(n+1)$ on $\mathbb{S}^n$ and $\mathbb{R}^{n+1}$. Building on this observation, we prove:
	 
\begin{thm}\label{thmB}
    In the setting of Theorem \ref{thmA} and its proof, assume further that $f$ is $L^2$-orthogonal to the linear functions. Then, for every sufficiently small $t\neq 0$, the isometry group of $(\mathbb{S}^n,g_t)$ is the stabilizer group
    \begin{equation*}
    	G_f = \{T\in O(n+1)\,|\, f\circ T=f\}.
    \end{equation*}
    Moreover, every isometry sends a member of the Zoll family of embedded minimal spheres in $(\mathbb{S}^n,g_t)$ into another member of the Zoll family. 
\end{thm}	    
    
	This result has several consequences of geometric interest. For instance, by choosing any smooth odd function on $\mathbb{S}^n$, $n\geq 3$, depending only on one coordinate of $\mathbb{R}^{n+1}$ and $L^2$-orthogonal to linear functions, we construct higher dimensional counterparts of Zoll's original spheres of revolution, in the context of minimal hypersufaces theory:
	
\begin{thm}\label{thmC}
	In all dimensions $n\geq 3$, there exist non-trivial $O(n)$-invariant spheres in $\mathbb{R}^{n+1}$ that contain a Zoll family of codimension one embedded minimal spheres which is itself preserved by the action of $O(n)$.
\end{thm}

	While O. Zoll's original metrics are explicit and can be far from being round, notice, however, that the metrics in Theorem \ref{thmC} are perturbations of canonical metric, and we did not find a simple formula describing the generating curve.
	 
	We also apply Theorem \ref{thmB} to obtain new Zoll metrics on $\mathbb{S}^2$:
	
	\begin{thm}\label{thmD}
  	  Let $G$ be a finite subgroup of $O(3)$ with $-Id\notin G$. Then there are embedded spheres in $\mathbb{R}^3$ whose induced metrics are Zoll and whose isometry groups are $G$.
	\end{thm}
	 
	 Notice that L. Green's theorem \cite{Gre} implies that the isometry group of a non-round Zoll metric on $\mathbb{S}^2$ cannot contain an orientation-reversing involution. Theorem \ref{thmD} confirms that this is the only obstruction in the case of finite groups (see Remark~\ref{rmkobst}).
	 
	The same strategy used to establish the equivariance of the construction of Theorem \ref{thmA} works for the constructions of Theorems A and E in \cite{ACM}. As an application, we show that it is possible to deform $(\mathbb{RP}^n,can)$ into new metrics that admit Zoll families of embedded minimal real projective hyperplanes:

\begin{thm}\label{thmE}
    Let $h$ be a smooth divergence-free traceless symmetric two-tensor on $(\RP^n,can)$, $n\geq 3$. Then there exists a smooth one-parameter family of Riemannian metrics $g_t$ on $\RP^n$, $|t|<\varepsilon$, with $g_t=can+th+o(t)$ as $t$ goes to zero, such that each $(\mathbb{RP}^n,g_t)$ contains a Zoll family of codimension one embedded minimal real projective spaces parametrized by $\mathbb{RP}^n$.
\end{thm}

	Recall that divergence-free, traceless, symmetric two-tensors with respect to $can$ parametrize infinitesimal directions that are transverse to the conformal class and to the orbit of $can$ under the action of diffeomorphism group by pull-back \cite{BerEbi}. In the conformal class of the canonical metric on $\mathbb{RP}^n$, $n\geq 3$, we do not know if there are metrics in $\mathcal{Z}$ aside from those with constant curvature (see, however, \cite{ACM}, Theorem B).
	
	Metrics on $\mathbb{RP}^n$ with minimal linear real projective hyperplanes were completely classified in \cite{ACM} (see Theorem 9.8). They form a subset $\mathcal{K}\subset \mathcal{Z}$, parametrized by an open cone in a vector space of finite dimension depending on $n$. Thus, for most choices of divergence-free, traceless tensors $h$ on $\mathbb{RP}^n$ (see their classification in \cite{BoucettaSpectredesLaplacien}), Theorem \ref{thmC} produces metrics on $\mathcal{Z}$ that are not isometric to any metric on $\mathcal{K}$ (see Remark \ref{rmkkilling}). \\

	As in \cite{ACM} and \cite{Guillemin}, all new constructions in this paper rely on some version of the Nash-Moser Implicit Function Theorem. We expect that our strategy to prove the equivariance of the constructions in the framework of tame Fr\'echet spaces of R. Hamilton \cite{HamiltonNashMoser} is adaptable to other contexts. For other constructions of Zoll-like objects with symmetries in magnetic dynamics, see the work of L. Asselle and G. Benedetti \cite{AssBen}, and of these two authors and M. Berti \cite{AssMaxBer}.

\subsection{Plan of the paper} After a preliminary Section \ref{section preliminary}, where we fix notations and perform some basic computations, we devote Section \ref{Section Construction} to the construction of the graphical deformations of $\mathbb{S}^n$ whose induced metrics lie in $\mathcal{Z}$. The construction closely follows the framework built in \cite{ACM}, and is an application of Nash-Moser Implicit Function Theorem with a quadratic error due to R. Hamilton (see \cite{HamiltonNashMoser}, Part III, Section 2). 

In Section \ref{subsection equivariance}, we first introduce some general notions pertaining to the theory of tame Fréchet spaces, and formulate the abstract equivariant versions of the aforementioned results of Hamilton (see Theorems \ref{thmequivarianceHam2} and \ref{thmequivariantHamIFTquadraticerror}). In Section \ref{subsectionapplication}, we check that the maps constructed in Section \ref{Section Construction} (and in \cite{ACM}) satisfy the hypotheses of these abstract results, yielding their equivariance with respect to the natural action of $O(n+1)$. Theorem \ref{thmE} follows, but the proof of Theorem \ref{thmB} requires additional work, which is carried out in Section \ref{subseccomputation}. We use rigidity results for isometric immersions in $\mathbb{R}^{n+1}$ of metrics with positive curvature to conclude that the intrinsic isometries of the metrics $g_t$ in Theorem \ref{thmB} are actually restrictions to $\iota_t(\mathbb{S}^n)$ of isometries of $\mathbb{R}^{n+1}$. This allows some control on sequences of isometries of the metrics $g_{t}$ for small $t$. As an application of Theorem \ref{thmB}, we prove Theorem \ref{thmC}.

Section \ref{sectionS2} contains the proof of Theorem \ref{thmD}. By Theorem \ref{thmB}, it is enough to find smooth odd functions on $\mathbb{S}^2$ with prescribed finite stabilizer subgroup of $O(3)$ not containing $-Id$. For each group in the classification of such groups, we exhibit the desired function explicitly.

Finally, in the Appendix, we collect some formulas for the second fundamental form of star-shaped embeddings, compare the metrics constructed in Theorem C of \cite{ACM} and Theorem \ref{thmB}, and discuss their isometry groups.

\section{Preliminaries and notations}\label{section preliminary}
	The basic concepts and notations introduced in this preliminary section are adapted from and motivated by \cite{ACM}.

	Let $can$ be the canonical metric on the sphere $\S^n=\{x\in\R^{n+1}:|x|=1\}$. The set of smooth odd functions with respect to the antipodal map $x\mapsto-x$ will be denoted by $C^\infty_{odd}(\S^n)$.

	Given $v\in \S^n$, $\Sigma_v\subseteq\S^n$ denotes the equator orthogonal to $v$. The equators are totally geodesic hypersurfaces of $(\S^n,can)$, and they are smoothly parametrized by the real projective space $\RP^n=\mathbb{S}^n/x\sim-x$, since $\Sigma_v=\Sigma_{-v}$. Hence, for $\sigma=[v]:=\{v,-v\}\in\RP^n$, we write $\Sigma_\sigma=\Sigma_v$. In order to maintain consistency throughout this work, we only use Greek letters for the elements of $\RP^n$.
	
	We want to smoothly deform $\mathbb{S}^n$ inside the Euclidean space $\mathbb{R}^{n+1}$, together with its standard Zoll family of totally geodesic equators $\{\Sigma_\sigma\}$, in such a way that the deformed family, inside the deformed embedding of $\mathbb{S}^n$, is a Zoll family of codimension one embedded minimal spheres with respect to the induced metric from $\mathbb{R}^{n+1}$. This is what we will achieve with the proof of Theorem \ref{thmA}. \\

	Given a smooth function $\psi\in C^\infty(\S^n)$, let $g_\psi$ denote the Riemannian metric on $\S^n$ given by 
	\begin{equation*}
		g_\psi:=e^{2\psi}(can+d\psi\otimes d\psi).
	\end{equation*}
	Equivalently, $g_\psi$ is the induced metric by the star-shaped embedding $\iota=\iota_\psi:\S^n\rightarrow\R^{n+1}$ given by 
	\begin{equation}\label{eqstarshapped2}
		\iota(x)=e^{\psi(x)}x.
	\end{equation}
	Clearly, $g_0=can$ and $\iota_0$ is the standard inclusion of $\mathbb{S}^n$ in $\mathbb{R}^{n+1}$.

	After observing that functions $\psi\in C^{\infty}(\mathbb{S}^n)$ correspond to metrics on $\mathbb{S}^n$ that are induced by star-shaped embeddings as in \eqref{eqstarshapped2}, let us describe how graphical perturbations of the family of equators are in one-to-one correspondence with certain functions on the unit tangent bundle of $\mathbb{S}^n$, that is, on the manifold 
	\begin{equation*}
		T_1\S^n=\{(p,v)\in\S^n\times\S^n\,|\,\inner{p}{v}=0\}.
	\end{equation*}
	
	Let $C^\infty_{*,odd}(T_1\S^n)$ be the set of functions $\Phi\in C^{\infty}(T_1\S^n)$ that are odd with respect to the second variable, that is, $\Phi(x,-v)=-\Phi(x,v)$. 
Given $\Phi\in C^\infty_{*,odd}(T_1\S^n)$, consider the map $F=F^\Phi:T_1\S^n\rightarrow\S^n$ defined by
$$F(x,v)=\cos(\Phi(x,v))x+\sin(\Phi(x,v))v.$$
Observe that $F_v^\Phi=F(\cdot,v):\Sigma_v\rightarrow\S^n$ is an embedding for $\Phi$ small ($|\Phi|<\pi/2$ is enough). Its image is denoted 
\begin{equation*}
	\Sigma_v(\Phi):=F^\Phi_v(\Sigma_v)\subseteq\S^n,
\end{equation*}
	Since $F_v^\Phi=F_{-v}^\Phi$, we have $\Sigma_v(\Phi)=\Sigma_{-v}(\Phi)$, so we also denote this set by $\Sigma_\sigma(\Phi)$ for $\sigma=[v]\in \mathbb{RP}^n$. We have $\{\Sigma_\sigma(0)\}=\{\Sigma_\sigma\}$, and the family $\{\Sigma_\sigma(\Phi)\}$ is the set of \textit{graphical deformations}, determined by $\Phi$, of the family of equators.
	
	Given any such family, for every $y=F_v^\Phi(x)\in \Sigma_v(\Phi)$ we define the variational vector 
\begin{equation}\label{eqvarivecton}
	n_v(y)=n^{\Phi}_v(y):=\frac{d}{dt}\Big|_{t=0}F^{\Phi+t}_v(x)=-\sin(\Phi_v(x))x+\cos(\Phi_v(x))v\in T_y\S^n.
\end{equation}
Notice that $n_v(y)\notin T_y(\Sigma_v(\Phi))$.
Intuitively, $n_v$ plays the role of a transversal vector to $\Sigma_v(\Phi)$, adjusted to the variation (beware that it might be non-orthogonal to $T_y(\Sigma_v(\Phi))$ with respect to a given metric $g_\psi$). 

Given a graphical deformation associated to $\Phi\in C^\infty_{*,odd}(T_1\S^n)$, we define the {\it incidence set} as 
\begin{equation*}
	[F(\Phi)]:=\{(y,\sigma)\in\S^n\times\RP^n\,|\,y\in\Sigma_\sigma(\Phi)\},
\end{equation*}
which is a compact hypersurface of $\S^n\times\RP^n$ for $\Phi$ small enough in the $C^1$ topology.
We also define the {\it dual hypersurface} associated to $p\in \mathbb{S}^n$ by
$$\Sigma_p^*(\Phi):=\{\sigma\in\RP^n\,|\,(p,\sigma)\in[F(\Phi)]\}.$$
For $\Phi$ close to zero in the $C^1$ topology, $\Sigma_p^*(\Phi)\subseteq\RP^n$ is a regular hypersurface and is diffeomorphic to $\RP^{n-1}$ (note that $\Sigma^*_p(0)=\Sigma_p/\{\pm Id\}$ is simply the real projective plane orthogonal to $p\in \mathbb{S}^n$).

	Assuming further that $\Phi$ is $C^2$-close to zero, the family $\Sigma_\sigma(\Phi)$ satisfies the topological axioms in the definition of a Zoll family of minimal hypersurfaces (see \cite{ACM}, Proposition 2.4).

	Thus, in this setup, we search for pairs $(\psi,\Phi)\in C^\infty(\S^n)\times C^\infty_{*,odd}(T_1\S^n)$, sufficiently close to $(0,0)$ in the smooth topology, such that $\Sigma_\sigma(\Phi)$ is a minimal hypersurface of $(\S^n,g_\psi)$ for all $\sigma\in\mathbb{RP}^n$. Equivalently, such that the family $\{\iota_\psi(\Sigma_\sigma(\Phi))\}$ is a Zoll family of minimal hypersurfaces on $\iota_\psi(\mathbb{S}^n)\subset \mathbb{R}^{n+1}$ with the induced metric by the embedding $\iota_\psi : \mathbb{S}^n\rightarrow \mathbb{R}^{n+1}$ as in \eqref{eqstarshapped2}.

\section{Construction}\label{Section Construction}

	The proof of Theorem \ref{thmA} follows closely the strategy developed in \cite{ACM}. In this section we adapt that framework to our setting, and prove the key formulas that allow the verification of the assumptions needed to apply the Nash-Moser-Hamilton theory of tame maps between tame Fr\`echet spaces.

\subsection{The Area operator}\label{subsectarea}
	The condition that $\Sigma_v(\Phi)$ is a minimal hypersurface in $(\S^n,g_\psi)$ is variational. Thus, we begin by analyzing the properties of the area functional in the class of graphical deformations of $\Sigma_v$. 
	
	Given $\psi\in C^\infty(\S^n)$ and $\Phi\in C^\infty_{*,odd}(T_1\S^n)$, the area functional is a function on $\mathbb{S}^n$ that computes, for every $v\in \mathbb{S}^n$,  
	\begin{equation*}
		\mathcal{A}_v(\psi,\Phi):=\text{Area}_{g_\psi}(\Sigma_v(\Phi))=\int_{\Sigma_v(\Phi)}dA_{g_\psi}.
	\end{equation*}
	Clearly, $\mathcal{A}_{-v}=\mathcal{A}_{v}$, so that we can think of $\mathcal{A}(\psi,\Phi)$ as a function on $\mathbb{RP}^n$.
	
	By a change of variables, this functional can be expressed as an integral over $(\Sigma_v,can)$. In order to compute it, we write $\Phi_v:\Sigma_v\rightarrow\R$ for the map $\Phi_v=\Phi(\cdot,v)$ on $\Sigma_v$, and let $\nabla^v$ be the gradient operator with respect to the canonical metric acting on smooth functions on $\Sigma_v$. Finally, given any smooth function $f:\S^n\rightarrow\R$, we write $f_v=f\circ F^{\Phi}_v:\Sigma_v\rightarrow\R$.
\begin{lem}\label{lemareafunctJ}
    The area functional is given by
    $$\mathcal{A}_v(\psi,\Phi)=\int_{\Sigma_v}J_{v}(\psi_v(x),\Phi_v(x),\nabla^v\psi_v(x),\nabla^v\Phi_v(x))dA_{can}(x),$$
    where $J_v$ is the smooth map, defined on the Whitney sum $\R^2\oplus T\Sigma_v\oplus T\Sigma_v$, given by the formula
    $$J_v(p,q,u,w)=e^{(n-1)p}\cos(q)^{n-3}\big[|u\wedge w|^2+\cos(q)^2(\cos(q)^2+|u|^2+|w|^2)\big]^{\frac{1}{2}}.$$
\end{lem}
\begin{proof}
    Since both $F_v^{*}dA_{g_\psi}$ and $dA_{can}$ are top forms of $\Sigma_v$, 
    \begin{equation}\label{pull-back area}
        F_v^{*}dA_{g_\psi}=J_vdA_{can},
    \end{equation}
    for a certain smooth function $J_v:\Sigma_v\rightarrow\R$.
    Changing variables on the integral, we only need to verify that $J_v$ is given by the formula presented in the statement of this Lemma.
    
    Suppose first that $\nabla^v\psi_v$ and $\nabla^v\Phi_v$ are linearly independent in a neighborhood $U\subseteq\Sigma_v$. 
    With respect to the canonical metric, consider a local orthogonal frame $\{e_i\}_{i=0}^{n-2}$ of $U$ where 
    $$e_0:=\nabla^v\psi_v, \quad e_1:=|\nabla^v\psi_v|^2\nabla^v\Phi_v-\inner{\nabla^v\psi_v}{\nabla^v\Phi_v}\nabla^v\psi_v,$$
    and the remaining $\{e_i\}_{i=2}^{n-2}$ are a local orthonormal frame of $\{e_0,e_1\}^{\perp}\subseteq T\Sigma_v$.
    Then, evaluating \eqref{pull-back area} in this frame, we have that 
    \begin{align}
        J_v|\nabla^v\psi_v|^2|\nabla^v\psi_v\wedge\nabla^v\Phi_v|&=dA_{g_\psi}(dF_v(e_0),dF_v(e_0),\ldots,dF_v(e_{n-2}))\nonumber\\
        &=\sqrt{\det(g_{\psi}(dF_v(e_i),dF_v(e_j))}.\label{Jv when both gradients are li}
    \end{align}
    Notice that 
    $$d\psi(dF_v(e_j))=d\psi_v(e_j)=\inner{\nabla^v\psi_v}{e_j},$$ 
    and, using the definition of \eqref{eqvarivecton},
    $$dF_v(e_j)=\inner{\nabla^v\Phi_v}{e_j}n_v+\cos(\Phi_v)e_j.$$ 
    Hence, the entries of the matrix $G_\psi=(g_\psi(e_i,e_j))_{ij}$ are
    \begin{align*}
        g_{\psi}(e_0,e_0)&=e^{2\psi_v}(\langle dF_ve_0,dF_ve_0\rangle+|\nabla^v\psi_v|^4)\\
        &=e^{2\psi_v}(\inner{\nabla^v\Phi_v}{\nabla^v\psi_v}^2+\cos(\Phi_v)^2|\nabla^v\psi_v|^2+|\nabla^v\psi_v|^4),\\
        g_{\psi}(e_0,e_1)&=e^{2\psi_v}\langle dF_ve_0,dF_ve_1\rangle \\
        &=e^{2\psi_v}(|\nabla^v\Phi_v\wedge\nabla^v\psi_v|^2\inner{\nabla^v\Phi_v}{\nabla^v\psi_v}),\\
        g_{\psi}(e_1,e_1)&=e^{2\psi_v}\langle dF_ve_1,dF_ve_1\rangle \\
        &=e^{2\psi_v}(|\nabla^v\Phi_v\wedge\nabla^v\psi_v|^4+\cos(\Phi_v)^2|\nabla^v\psi_v|^2|\nabla^v\Phi_v\wedge\nabla^v\psi_v|^2),\\
        g_{\psi}(e_i,e_j)&=e^{2\psi_v}g_0(dF_ve_i,dF_ve_j)=\delta_{ij}e^{2\psi_v}\cos(\Phi_v)^2,\quad\text{for }j\geq2.
    \end{align*}
    Using these formulas, the determinant of $G_\psi$ is computed, and the desired formula for $J_v$ is deduced from \eqref{Jv when both gradients are li}, at least at each open neighborhood $U\subset \Sigma_v$ where  $\nabla^v\Phi_v$ and $\nabla^v\psi_v$ are linearly independent. 

    Suppose now that $U\subseteq\Sigma_v$ is an open subset where $\nabla^v\Phi_v$ and $\nabla^v\psi_v$ span a line. Consider a local orthogonal frame $\{e_i\}_{i=0}^{n-2}$ of $U$ with $\nabla^v\Phi_v, \nabla^v\psi_v\in\text{span}\{e_0\}$. Similar computations of the determinant of $G=(g_\psi(e_i,e_j))_{ij}$ yield a direct verification that the formula for $J_v$ works in this case as well.
    
    Finally, consider the situation in which $U\subseteq\Sigma_v$ is an open subset where $\nabla^v\Phi_v=\nabla^v\psi_v=0$. 
    In this case, $\phi_v$ and $\psi_v$ are constant, so we simply scale the metric of $\Sigma_v\cap U$ by the factor $e^{\psi_v}\cos(\phi_v)$.
    It is straightforward to verify that the formula of $J_v$ also holds in this situation.
    
    These three cases show that the formula for $J_v$ holds on a dense subset of $\Sigma_v$, so they coincide everywhere by continuity. 
\end{proof}

	If $\Sigma_v(\Phi)$ is a minimal hypersurface of $(\S^n,g_\psi)$, then it is a critical point of the area functional. In particular, in this case the pair $(\psi,\Phi)$ is such that $\mathcal{A}(\psi,\Phi)$ has a zero derivative in its second entry. Before analyzing what does this condition mean, we compute :
	
\begin{lem}\label{lemeulerlagrangeoperator}
	For every $v\in \mathbb{S}^n$, 
\begin{equation*}
	\frac{d}{dt}\Big|_{t=0}\mathcal{A}_v(\psi,\Phi+t\phi)=\int_{\Sigma_v}\mathcal{H}_v(\psi,\Phi)\phi_vdA_{can},\quad\forall\phi\in C^\infty_{*,odd}(T_1\S^n),
\end{equation*}
where 
\begin{align}\label{eqformulaofH}
\mathcal{H}_v(\psi,\Phi)=&D_1J_v(\psi_v,\Phi_v,\nabla^v\psi_v,\nabla^{v}\Phi_v)\partial_n\psi_v \nonumber \\
    & 
    +D_2J_v(\psi_v,\Phi_v,\nabla^v\psi_v,\nabla^{v}\Phi_v ) \nonumber\\
    &-div_{(\Sigma_v,can)}\big(D_3J_v(\psi_v,\Phi_v,\nabla^v\psi_v,\nabla^{v}\Phi_v)\big)\partial_n\psi_v
    \nonumber \\
    & -div_{(\Sigma_v,can)}\big(D_4J_v(\psi_v,\Phi_v,\nabla^v\psi_v,\nabla^{v}\Phi_v)\big).
\end{align}
Here, $\partial_n\psi_v$ is the {\it derivative of $\psi$ in the $n_v$ direction}, and is given by
\begin{equation*}
	\partial_n\psi_v:=\inner{\nabla\psi}{n_v}\circ F_v:\Sigma_v\rightarrow\R,
\end{equation*}
where $\nabla\psi$ is the gradient of $\psi$ on $(\mathbb{S}^n,can)$, and  $n_v$ is the transverse vector field defined in \eqref{eqvarivecton}.
\end{lem}
\begin{proof}
Given $\phi\in C^\infty_{*,odd}(T_1\S^n)$, consider the derivative \begin{align*}
    \frac{d}{dt}\Big|_{t=0}\mathcal{A}_v(\psi,\Phi+t\phi &) =  \int_{\Sigma_v} \hspace{-1mm}\Big[D_1J_v(\psi_v,\Phi_v,\nabla^v\psi_v,\nabla^{v}\Phi_v)\frac{d}{dt}\Big|_{t=0}\hspace{-1mm}(\psi\circ F_v^{\Phi+t\phi}) \\
    & +D_2J_v(\psi_v,\Phi_v,\nabla^v\psi_v,\nabla^{v}\Phi_v)\phi_v\\
    & +\Big\langle D_3J_v(\psi_v,\Phi_v,\nabla^v\psi_v,\nabla^{v}\Phi_v),\frac{d}{dt}\Big|_{t=0}\hspace{-1mm}\big(\nabla^v(\psi\circ F_v^{\Phi+t\phi})\big)\Big\rangle\\
    &+\inner{D_4J_v(\psi_v,\Phi_v,\nabla^v\psi_v,\nabla^{v}\Phi_v)}{\nabla^v\phi_v}\Big]
    dA_{can},
\end{align*}
where $\phi_v=\phi(\cdot,v)$, and $D_iJ_v$ denotes the derivative of $J_v$ with respect to the $i^{th}$-coordinate. A computation gives
\begin{align}
    D_1J_v(p,q,u,w)&=(n-1)J_v(p,q,u,w), \label{eqJ1}\\
    D_2J_v(p,q,u,w)&=\tan(q)J_v(p,q,u,w)\bigg((n-3)+\nonumber\\
    &\quad\quad\quad\quad\quad\quad+\frac{\cos(q)^2(2\cos(q)^2+|u|^2+|w|^2)}{|u\wedge w|^2+\cos(q)^2(\cos(q)^2+|u|^2+|w|^2)}\bigg), \label{eqJ2}\\
    D_3J_v(p,q,u,w)&=\frac{J_v(p,q,u,w)((|w|^2+\cos(q)^2)u-\inner{w}{u}w)}{|u\wedge w|^2+\cos(q)^2(\cos(q)^2+|u|^2+|w|^2)}, \label{eqJ3}\\
    D_4J_v(p,q,u,w)&=D_3J_v(p,q,w,u). \label{eqJ4}
\end{align}
Since
\begin{equation*}
	\frac{d}{dt}\Big|_{t=0}(\psi\circ F_v^{\Phi+t\phi})=\partial_n\psi_v\phi_v \quad \text{and}\quad \frac{d}{dt}\Big|_{t=0}\big(\nabla^v(\psi\circ F_v^{\Phi+t\phi})\big)=\nabla^v(\partial_n\psi_v\phi_v),
\end{equation*}
formula \eqref{eqformulaofH} is obtained after applying Stokes' theorem.
\end{proof}

	The function $\mathcal{H}(\psi,\Phi)$ belongs to $C^\infty_{*,odd}(T_1\S^n)$. It depends on $\Phi$ and $\psi$ up to their second-order derivatives. While it is not exactly the mean curvature of $\Sigma_\sigma(\Phi)$ in $(\mathbb{S}^n,g_\psi)$, we have
	\begin{lem}\label{lemH=0meansmin}
		$\mathcal{H}(\psi,\Phi)\equiv 0$ if and only if $\Sigma_v(\Phi)$ is a minimal hypersurface in $(\mathbb{S}^n,g_\psi)$ for every $v\in \mathbb{S}^n$.
	\end{lem}

\begin{proof} Let $N_v$ be a normal unit vector field on $\Sigma_v(\Phi)$, with respect to the metric $g_\psi$, and let $\hat{\mathcal{H}}_v(y)$ be the mean curvature of $\Sigma_v(\Phi)$ in $(\S^n,g_\psi)$ at the point $y$ with respect to $N_v$ (we use the convention that $\vec{H}=-HN$ is the mean curvature vector). Notice that $c(y):=g_\psi(n_v(y),N_v(y))\neq 0$, since $n_v(y)\notin T_y(\Sigma_v(\Phi))$.
By definition of $\mathcal{A}_v$ and the first variation of area formula,
\begin{align*}
    \frac{d}{dt}\Big|_{t=0}\mathcal{A}_v(\psi,\Phi+t\phi)&=\int_{\Sigma_v(\Phi)}\hat{\mathcal{H}}_v(y)g_{\psi}\Big(\frac{d}{dt}\Big|_{t=0}(F_v^{\Phi_v+t\phi_v}),N_y\Big)dA_{g_\psi}(y)\\
    &=\int_{\Sigma_v(\Phi)}\hat{\mathcal{H}}_v(y)g_{\psi}\Big(\phi_v(F_v^{-1}(y))n_y,N_y\Big)dA_{g_\psi}(y)\\
    &=\int_{\Sigma_v(\Phi)}\hat{\mathcal{H}}_v(y)c(y)\phi_v(F_v^{-1}(y))dA_{g_\psi}(y)\\
    &=\int_{\Sigma_v}\hat{\mathcal{H}}_v(F_v(x))c(F_v(x))\phi_v(x)J_vdA_{can}(x).
\end{align*}

Hence
\begin{equation}\label{eqcalHh}
	\mathcal{H}_v(\psi,\Phi)=(\hat{\mathcal{H}}_v\circ F_v)(c\circ F_v)J_v,
\end{equation}
from which it follows that $\Sigma_v(\Phi)\subseteq(\S^n,g_\psi)$ has zero mean curvature if and only if $\mathcal{H}_v(\psi,\Phi)\equiv0$.
\end{proof}

	In view of Lemma \ref{lemH=0meansmin}, from now on we look for  pairs $(\psi,\Phi)\in C^\infty(\S^n)\times C^\infty_{*,odd}(T_1\S^n)$, sufficiently close to $(0,0)$ in the $C^2$ topology, such that $\mathcal{H}(\psi,\Phi)$ vanish identically. In this case, $\{\Sigma_\sigma(\Phi)\}$, $\sigma\in \mathbb{RP}^n$, is a Zoll family of codimension one embedded minimal spheres in $(\mathbb{S}^n,g_\psi)$.

\subsection{The Funk transform and its dual}\label{subsectfunk}

In the previous section, we analyzed the partial derivative $D_2\mathcal{A}(\psi,\Phi)$ in order to characterize the condition that all $\Sigma_v(\Phi)\subseteq(S^n,g_\psi)$ are minimal hypersurfaces in $(\mathbb{S}^n,g_\psi)$. In this section, we analyze $D_1\mathcal{A}(\psi,\Phi)$, which is crucial for understanding when there are one-parameter families $(\psi_t,\Phi_t)$ such that $\mathcal{H}(\psi_t,\Psi_t)\equiv 0$ for every small $t$.

We recognize this derivative as a generalized Funk transform, similar to the ones considered in \cite{ACM}, Section 9.5. In fact, a way to compute $D_1\mathcal{A}(\psi,\Phi)$ is to observe that
	\begin{equation*}
		(D_1\mathcal{A}(\psi,\Phi)\cdot f)(v):=\frac{1}{2}\int_{\Sigma_v(\Phi)}\text{tr}_{(\Sigma_v(\Phi),g_\psi)}\Big(\frac{d}{dt}\Big|_{t=0}g_{\psi+tf}\Big)dA_{g_{\psi}}.
	\end{equation*}
	This is nothing but the tensor transform considered in \cite{ACM}, Section 9.5, computed at the symmetric two-tensor 	
	\begin{equation*}
		h=2fg_\psi+e^{2\psi}(df \otimes d\psi+ d\psi\otimes df). 
	\end{equation*}
	In particular, 
	\begin{equation*}
		D_1\mathcal{A}(0,0)\cdot f: v\in \mathbb{S}^n \mapsto  (n-1)\int_{\Sigma_v}f\,dA_{can} \in \mathbb{R}.
	\end{equation*}
	is the standard Funk transform (up to the $(n-1)$ factor). The kernel and the range of this transform are well known to be the odd and even smooth functions on $\mathbb{S}^n$, respectively (\textit{cf.} Appendix of \cite{ABS}).
	
	At the end of this section, we show that the transform $D_1\mathcal{A}(\psi,\Phi)$, viewed as a map taking values in $C^{\infty}(\mathbb{RP}^n)$, has a right-inverse, as soon as $(\psi,\Phi)$ sufficiently close to $(0,0)$ in the $C^{3n+3}$ topology. This fact will be crucial for the proof of Theorem \ref{thmA}. 

Since more explicit formulae for $D_1\mathcal{A}(\psi,\Phi)\cdot f$ are required for our analysis, we compute:

\begin{lem}\label{lemD1A}
    There exists a smooth function $\mathcal{Q}_v=\mathcal{Q}_v(\psi,\Phi):\Sigma_v(\Phi)\rightarrow\R$ such that $\mathcal{Q}_v=\mathcal{Q}_{-v}$ and
    \begin{equation*}
    	\frac{d}{dt}\Big|_{t=0}\mathcal{A}_v(\psi+tf,\Phi)=\int_{\Sigma_v(\Phi)}f(y)\mathcal{Q}_v(y)dA_{g_{\psi}}(y).
    \end{equation*}
    Moreover, if $(\psi,\Phi)$ is $C^{k+2}$-close to $(0,0)$, then the function $\mathcal{Q}=\mathcal{Q}(\psi,\Phi)\in C^\infty_{*,odd}(T_1\S^n)$ given by 
    \begin{equation*}
    	\mathcal{Q}(x,v)=\mathcal{Q}_v(F_v(x)),
    \end{equation*}
    is $C^k$-close to the constant function $(n-1)$.
\end{lem}
\begin{proof}
    We have
    \begin{align*}
    \frac{d}{dt}\Big|_{t=0}\mathcal{A}_v(\psi+tf,\Phi) & =\int_{\Sigma_v}\big(D_1J_v(\psi_v,\Phi_v,\nabla^v\psi_v,\nabla^v\Phi_v)f_v \\
     &\quad +\inner{D_3J_{v}(\psi_v,\Phi_v,\nabla^v\psi_v,\nabla^v\Phi_v)}{\nabla^vf_v}\big)dA_{can} \\
     &=\int_{\Sigma_v}\Big(D_1J_v(\psi_v,\Phi_v,\nabla^v\psi_v,\nabla^v\Phi_v)\\
     &\quad -\text{div}_{\Sigma_v}\big(D_3J_{v}(\psi_v,\Phi_v,\nabla^v\psi_v,\nabla^v\Phi_v)\big)\Big)f_v dA_{can}.
    \end{align*}
    To simplify the notation, let $J_v:\Sigma_v\rightarrow\R$ be the function given by $J_v=J_v(\psi_v,\Phi_v,\nabla^v\psi_v,\nabla^v\Phi_v)$. Using the explicit formula of $J_v$ deduced in Lemma \ref{lemareafunctJ} and its derivatives computed in Lemma \ref{lemeulerlagrangeoperator}, we obtain 
    \begin{equation}\label{eqD1J}
    	D_1J_v(\psi_v,\Phi_v,\nabla^v\psi_v,\nabla^v\Phi_v)=(n-1)J_v,
    \end{equation}
    and
    \begin{equation}\label{eqD3J}
    	D_3J_v(\psi_v,\Phi_v,\nabla^v\psi_v,\nabla^v\Phi_v)=J_vX_v,
    \end{equation}
    where $X_v$ is the vector field tangent to $\Sigma_v$ given by 
    \begin{equation*}
    	X_v:=\frac{(\cos(\Phi_v)^2+|\nabla^v\Phi_v|^2)\nabla^v\psi_v-\inner{\nabla^v\Phi_v}{\nabla^v\psi_v}\nabla^v\Phi_v}{|\nabla^v\Phi_v\wedge\nabla^v\psi_v|^2+\cos(\Phi_v)^2(\cos(\Phi_v)^2+|\nabla^v\Phi_v|^2+|\nabla^v\psi_v|^2)}.
    \end{equation*}
    Hence, 
    \begin{equation}\label{eqdivJ3}
    	\text{div}_{\Sigma_v}\big(D_3J(\psi_v,\Phi_v,\nabla^v\psi_v,\nabla^v\Phi_v)\big)=dJ_v(X_v)+J_v\text{div}_{\Sigma_v}X_v.
   	\end{equation}
   	It follows from \eqref{eqD3J} that $dJ_v(X_v)=B_vJ_v$, where
    \begin{equation}\label{eqBv}
        B_v = (n-1)d\psi_v(X_v)+(n-3)\tan(\Phi_v)d\Phi_v(X_v) +d\xi(X_v),
    \end{equation}
    and 
    $$\xi=\frac{1}{2}\log\big(|\nabla^v\Phi_v\wedge\nabla^v\psi_v|^2
    +\cos(\Phi_v)^2(\cos(\Psi_v)^2 +|\nabla^v\Phi_v|^2+|\nabla^v\psi_v|^2)\big).$$
The function $B_v$ depends on $\psi$ and $\Phi$ up to their second derivatives. Combining \eqref{eqD1J}, \eqref{eqdivJ3} and \eqref{eqBv}, 
\begin{equation}\label{eqD1Abefore}
    \frac{d}{dt}\Big|_{t=0}\mathcal{A}_v(\psi+tf,\Phi) 
    =\int_{\Sigma_v}\big(n-1-B_v-\text{div}_{\Sigma_v}(X_v)\big)f_vJ_vdA_{can}.
\end{equation}
Let $\mathcal{Q}_v:\Sigma_v(\Phi)\rightarrow\R$ be defined by
\begin{equation*}
	\mathcal{Q}_v(y):=(n-1)-B_v\circ F_v^{-1}(y)-\text{div}_{\Sigma_v}(X_v)\circ F_v^{-1}(y).
\end{equation*}
Since $F^*_vdA_{g_\psi}=J_vdA_{can}$ (see Lemma \ref{lemareafunctJ}), by changing variables in \eqref{eqD1Abefore} ($y=F_v(x)$, $\Sigma_v(\Phi)=F_v(\Sigma_v)$), the first part of the Lemma is proven.

Finally, observe that $\mathcal{Q}=\mathcal{Q}(\psi,\Phi):T_1\S^n\rightarrow\R$ is given by
$$\mathcal{Q}(x,v)=\mathcal{Q}_v(F_v(x))=(n-1)-B_v(x)-\text{div}_{\Sigma_v}(X_v)(x),$$
which depends on $\psi$ and $\Phi$ up to derivatives of second order, and is even in the variable $v$. Since $\mathcal{Q}(0,0)\equiv n-1$, the last assertion follows.
\end{proof}
	
	Given $(\psi,\Phi)\in C^{\infty}(\S^n)\times C^{\infty}_{*,odd}(T_1\S^n)$ sufficiently close to $(0,0)$ in the smooth topology, it is convenient to consider the area functional $\mathcal{A}(\psi,\Phi)$ as a map on $C^{\infty}(\mathbb{RP}^n)$. We then define the \textit{generalized Funk transform} $\mathcal{F}(\psi,\Phi):C^\infty(\S^n)\rightarrow C^\infty(\RP^n)$ by setting
	\begin{equation*}
		\big(\mathcal{F}(\psi,\Phi)(f)\big)(\sigma):=(D_1\mathcal{A}(\psi,\Phi)\cdot f)(\sigma),
	\end{equation*}
	for every $\sigma=[v]\in \mathbb{RP}^n$. Lemma \ref{lemD1A} shows that
	\begin{equation*}
    \big(\mathcal{F}(\psi,\Phi)(f)\big)(\sigma)=\int_{\Sigma_\sigma(\Phi)}f(y)\mathcal{Q}_v(y)dA_{g_{\psi}}(y).
	\end{equation*}
	Another useful expression is 
\begin{equation*}
	\big(\mathcal{F}(\psi,\Phi)(f)\big)(\sigma) = \int_{\Sigma_v(\Phi)}f(y) \mathcal{Q}_v(y)T_v(y)dA_{can}(y),
\end{equation*}
for some $T_v=T_{-v}\in C^{\infty}(\Sigma_v)$. Notice that we may view both $\mathcal{Q}$ and $T$ as smooth functions on the incidence set $[F(\Phi)]$.

Similarly to what was done in Lemma \ref{lemareafunctJ}, it is straightforward to compute $T_v$:
\begin{lem}\label{lemmrelationvolumeforms}
    The volume forms of $\Sigma_v(\Phi)$ induced by canonical round metric $g_0$ and $g_\psi$ at $y=F_v(x)\in \Sigma_v(\Phi)$ are related by
    $$dA_{g_{\psi}}=T_v(\psi_v(x),\Phi_v(x),\nabla^v\psi_v(x),\nabla^v\Phi_v(x))dA_{can},$$
    where $T_v$ is a smooth function defined on the Whitney sum $\R^2\oplus T\Sigma_v\oplus T\Sigma_v$ and given by
    $$T_v(p,q,u,w)\hspace{-0.5mm}=\hspace{-0.5mm}e^{(n-1)p}\cos(q)^{-1}\hspace{-0.5mm}\bigg(\frac{|u\wedge w|^2+\cos(q)^2(\cos(q)^2+|u|^2+|w|^2)}{|w|^2+\cos(q)^2}\hspace{-0.5mm}\bigg)^{\frac{1}{2}}.$$
    Moreover, if $(\psi,\Phi)$ is $C^{k+1}$-close to $(0,0)$ then $T:T_1\S^n\rightarrow\R$ given by $T(x,v)=T_v(F_v(x))$ is $C^k$-close to the constant function $1$.
\end{lem}

	In order to apply the Nash-Moser implicit function theorem, it will be crucial to obtain a right-inverse of $\mathcal{F}(\psi,\Phi)$, when $(\psi,\Phi)$ is sufficiently close to $(0,0)$ in the smooth topology, that depends on $(\psi,\Phi)$ in a controlled way. Recall that right-inverses of linear maps, when they exist, are not necessarily unique. However, there is a canonical choice of a right-inverse $R$ for a linear map $A:\mathbb{R}^k\rightarrow \mathbb{R}^m$ such that $A\circ A^*:\mathbb{R}^m\rightarrow \mathbb{R}^m$ is invertible, where $A^*:\mathbb{R}^m\rightarrow \mathbb{R}^k$ is the adjoint map. Namely, the map $R=A^*\circ (A\circ A^*)^{-1}: \mathbb{R}^m\rightarrow \mathbb{R}^k$. This suggests the study of the dual transform of $\mathcal{F}(\psi,\Phi)$, that is, the map 
$\mathcal{F}^{*}(\psi,\Phi):C^\infty(\RP^n)\rightarrow C^\infty(\S^n)$
defined so that
\begin{equation*}
	\int_{\RP^n}\big(\mathcal{F}(\psi,\Phi)(f)\big)(\sigma)g(\sigma)dA_{can}(\sigma)=\int_{\S^n}f(y)\big(\mathcal{F}^*(\psi,\Phi)g\big)(y)dA_{can}(y),
\end{equation*}
holds for every $f\in C^\infty(\S^n)$ and $g\in C^\infty(\RP^n)$. 

Using the co-area formula for Riemannian submersions in an analogous way as done in \cite{ACM}, Section 7.3, (and considering the maps $\mathcal{Q}$ and $T$ from Lemmas \ref{lemD1A} and \ref{lemmrelationvolumeforms} as smooth maps of the incidence set $[F(\Phi)]$), we obtain
\begin{multline*}
    \int_{\RP^n}\big(\mathcal{F}(\psi,\Phi)(f)\big)(\sigma)g(\sigma)dA_{can}(\sigma) = \\
    = \int_{\RP^n}\Big(\int_{\Sigma_\sigma(\Phi)}\mathcal{Q}(y,\sigma)f(y)T(y,\sigma)dA_{can}(y)\Big)g(\sigma)dA_{can}(\sigma) = \\
    =\int_{\S^n}f(y)\Big(\int_{\Sigma_y^{*}(\Phi)}g(\sigma)\mathcal{Q}(y,\sigma)T(y,\sigma)U(\Phi)(y,\sigma)dA_{can}(\sigma)\Big)dA_{can}(y),
\end{multline*}
where $U(\Phi)$ is a certain smooth function, defined on the incidence set $[F(\Phi)]$, whose explicit formula was computed in Proposition 7.3 of \cite{ACM}. 

Thus, we define the \textit{dual generalized Funk transform} of $g\in C^{\infty}(\mathbb{RP}^n)$ as the function $\mathcal{F}^*(\psi,\Phi)g\in C^{\infty}(\S^n)$ given by
\begin{equation*}\label{eqF*}
	\big(\mathcal{F}^*(\psi,\Phi)g\big)(y):=\int_{\Sigma_y^{*}(\Phi)}g(\sigma)\mathcal{Q}(y,\sigma)T(y,\sigma)U(\Phi)(y,\sigma)dA_{can}(\sigma)
\end{equation*}
for every $y\in \mathbb{S}^n$.  \\

Let $\Delta\subseteq\RP^n\times\RP^n$ be the diagonal. For  every $[u]$, $[v]\in\RP^n\times\RP^n\setminus\Delta$, consider 
\begin{equation}\label{eqkernel}
	K(\psi,\Phi)([u],[v]):=\int_{\Sigma_u(\Phi)\cap\Sigma_v(\Phi)}\frac{\mathcal{Q}(y,u)T(y,u)\mathcal{Q}(y,v)T(y,v)}{\sqrt{1-\inner{N(\Phi)(y,u)}{N(\Phi)(y,v)}^2}}dA_{can}(y).
\end{equation}
where $N(\Phi):[F(\Phi)]\rightarrow\S^n$ associates to $(y,v)$ the unit normal vector (in the canonical metric) of $\Sigma_v(\Phi)$ in $(\S^n,can)$ at the point $y$. Notice that, in fact, the above formula is well defined because it does not change under $u\mapsto -u$ or $v\mapsto -v$.

Following the deduction in \cite{ACM}, Section 7.4, it is possible to check that
\begin{align*}
	\big(L(\psi,\Phi)(f)\big)(\sigma)& := \big(\mathcal{F}(\psi,\Phi)\circ\mathcal{F}^*(\psi,\Phi)(f)\big)(\sigma)\\
	& =\int_{\RP^n}f(\tau)K(\psi,\Phi)(\sigma,\tau)dA_{can}(\tau).
\end{align*}

The kernel $K(\psi,\Phi)$ has the same structural properties as the kernel used in \cite{ACM} (see Proposition 7.4). Therefore, Proposition 6.1 of \cite{ACM} can be applied. It follows that $L(\psi,\Phi)$ is a pseudo-differential operator of order $1-n$. Moreover, since $L(0,0)$ is elliptic and invertible (see Lemma A.1 in \cite{ACM}), $L(\psi,\Phi)$ is also elliptic and invertible for every $(\psi,\Phi)$ sufficiently small in the $C^{3n+3}$ topology.

In summary, the map
\begin{equation}\label{eqR}
	\mathcal{R}(\psi,\Phi):=\mathcal{F}^*(\psi,\Phi)\circ\big(\mathcal{F}(\psi,\Phi)\circ\mathcal{F}^*(\psi,\Phi)\big)^{-1}:C^\infty(\RP^n)\rightarrow C^\infty(\S^n).
\end{equation}
is a right-inverse of $\mathcal{F}(\psi,\Phi)$, as soon as $(\psi,\Phi)$ is in a sufficiently close $C^{3n+3}$ neighborhood of $(0,0)$.

\subsection{The Nash-Moser framework} \label{subsectnashmoserframework}
While the equation $\mathcal{H}(\psi,\Phi)\equiv 0$ characterizes the pairs $(\psi,\Phi)$ such that $\Sigma_v(\Phi)$ is minimal in $(\S^n,g_\psi)$ for every $v\in\S^n$, we will give an equivalent criterion for this property to hold. The usefulness of this criterion is that it can be expressed as distinguishing the zero level set of a smooth tame map $\Lambda$ between tame Fréchet spaces that satisfy the assumptions of the Nash-Moser Implicit Function Theorem with quadratic error due to R. Hamilton \cite{HamiltonNashMoser} (see Part III, Theorem 3.3.1). This criterion was key for the deformation theorems proven in \cite{ACM}, and will be central for us here as well. \\

	As a first step, and following Section 3.3 of \cite{ACM}, let us introduce a splitting of $C^\infty_{*,odd}(T_1\S^n)$ as a direct sum of two spaces naturally adapted to our problem.

	Let $\Omega^1_{even}(\S^n)$ be the set of smooth one-forms on $\S^n$ that are invariant under the antipodal map. The \textit{center map} is the map $C:C^{\infty}_{*,odd}(T_1\S^n)\rightarrow\Omega^1_{\text{even}}(\S^n)$ given by
	\begin{equation*}
		\big(C(\Phi)\big)_v(u)=\int_{\Sigma_v}\Phi_v(x)\inner{u}{x}dA_{can}(x)
	\end{equation*}
for all $u\in T_v\mathbb{S}^n$. The center map measures if the functions $\{\Phi_v\}_v$ are $L^2$-orthogonal to the linear functions on $(\Sigma_v,can)$. The kernel of $C$ is denoted by $C_{0,odd}^\infty(T_1\S^n)$.

	The center map has an explicit right-inverse $j:\Omega^1_{\text{even}}(\S^n)\rightarrow C^{\infty}_{*,odd}(T_1\S^n)$ (see \cite{ACM}, Proposition 3.5). The maps $C$ and $j$ then determine a splitting
	\begin{equation}\label{eqdecompositionC*odd}
	C^{\infty}_{*,odd}(T_1\S^n)=C_{0,odd}^\infty(T_1\S^n)\oplus j\big(\Omega^1_{\text{even}}(\S^n)\big).
	\end{equation}
	We show that $\mathcal{H}(\psi,\Phi)\equiv0$ is equivalent to two conditions: the constancy of the area of the hypersurfaces $\Sigma_v(\Phi)$ in $(\mathbb{S}^n,g_\psi)$, and the vanishing of the component of $\mathcal{H}(\psi,\Phi)$ in $C^\infty_{0,odd}(T_1\S^n)$:

\begin{prop}\label{propcriterion}
	The following assertions are equivalent:
	\begin{itemize}
		\item[$i)$] $\mathcal{H}(\psi,\Phi)\equiv 0$.
		\item[$ii)$] $\mathcal{A}(\psi,\Phi)$ is constant and $\mathcal{H}(\psi,\Phi)\equiv jC\mathcal{H}(\psi,\Phi)$.
	\end{itemize}
\end{prop}
\begin{proof}
	The same as the proof of Proposition 3.8 in \cite{ACM}, with obvious notational changes.
\end{proof}

Consider the spaces
\begin{equation}\label{F and H}
    F:=C^\infty(\S^n)\times C^\infty_{0,odd}(T_1\S^n)\text{ and }H:=C^\infty_0(\RP^n)\times C^\infty_{0,odd}(T_1\S^n),
\end{equation}
where $C^\infty_0(\RP^n)$ is the set of zero average smooth functions on $(\RP^n,can)$. Given a sufficiently small neighborhood of $(0,0)$ in $F$, consider the map $\Lambda=(\Lambda_1,\Lambda_2):(U\subseteq F)\rightarrow H$ be given by 
\begin{equation}\label{Lambda1 and Lambda2}
\begin{split}
	\Lambda_1(\psi,\Phi)& =\mathcal{A}(\psi,\Phi)-\frac{1}{\text{Vol}(\RP^n,can)}\int_{\RP^n}\mathcal{A}_\sigma(\psi,\Phi)dA_{can}(\sigma), \\
	\Lambda_2(\psi,\Phi) &=\mathcal{H}(\psi,\Phi)-jC(\mathcal{H}(\psi,\Phi)). 
\end{split}
\end{equation}
By Proposition \ref{propcriterion}, the zeroes of $\Lambda$ are precisely the pairs $(\psi,\Phi)\in F$ such that $\mathcal{H}(\psi,\Phi)\equiv 0$. 

	The restriction on the domain of the variable $\Phi$ in the definition of $\Lambda$ is important for the following reason. Given $(\psi,\Phi)\in C^\infty(\S^n)\times C^\infty_{*,odd}(T_1\S^n)$, consider the operator 
\begin{align*}
    \mathcal{P}(\psi,\Phi):C^\infty_{0,odd}(T_1\S^n)&\rightarrow C^\infty_{0,odd}(T_1\S^n),\\
    \phi\rightarrow \frac{d}{dt}&\Big|_{t=0}\Big(\mathcal{H}(\psi,\Phi+t\phi)-jC\mathcal{H}(\psi,\Phi+t\phi)\Big).
\end{align*}
It is straightforward to compute (using, for instance, \eqref{eqcalHh}) that
$$\big(\mathcal{P}(0,0)(\phi)\big)(x,v)=-\Delta_{\Sigma_v}\phi_v(x)-(n-1)\phi_v(x).$$
Since the kernel of the operator $-\Delta_{\Sigma_v}-(n-1)$ on $(\Sigma_v,can)$ is precisely the linear functions, we conclude that $C^\infty_{0,odd}(T_1\S^n)$ is nothing but the set of functions $\Phi\in C^\infty_{0,odd}(T_1\S^n)$ such that every $\Phi_v\in C^\infty(\S^n)$ lies in the $L^2$-orthogonal complement of the kernel of the Jacobi operator of $\Sigma_v$ inside $(\mathbb{S}^n,can)$, that is, the set of linear functions restricted to $\Sigma_v$.

	By standard elliptic theory, $\mathcal{P}(0,0)$ is invertible. Moreover, $\mathcal{P}(\psi,\Phi)$ is invertible for $(\psi,\Phi)$ sufficiently small in the $C^3$ norm (see \cite{ACM}, Proposition 3.3). We denote by 
\begin{equation}\label{eqS}
	\mathcal{S}(\psi,\Phi):C^\infty_{0,odd}(T_1\S^n)\rightarrow C^\infty_{0,odd}(T_1\S^n)
\end{equation}
the inverse of $\mathcal{P}(\psi,\Phi)$.

We are now ready to state the main theorem of this section, which has Theorem \ref{thmA} as a consequence.

\begin{thm}\label{thmGamma}
    There exists a neighborhood $W$ of $0\in \ker(D\Lambda(0,0))$ and a smooth tame map 
    $$\Gamma=(\Gamma_1,\Gamma_2):\ker(D\Lambda(0,0))\cap W\rightarrow F,$$ 
    such that $\Gamma(0)=0$, $Im(\Gamma)\subseteq\Lambda^{-1}(0)$, and $D\Gamma(0,0)v=v$ for every $v\in\ker D\Lambda(0,0)$.
\end{thm}
\begin{proof}
    This is a direct adaptation of \cite{ACM}, Sections 5 and 8. For the reader's convenience, we highlight the main steps.

    Let $V=(V_1,V_2):(U\subseteq F)\times H\rightarrow F$ and $Q=(Q_1,Q_2):(U\subseteq F)\times H\times H\rightarrow H$ be defined by the same formulas of Section 5.2 of \cite{ACM}, but with the functions $\mathcal{A}(\psi,\Phi)$ as in Lemma \ref{lemareafunctJ}, $\mathcal{H}(\psi,\Phi)$ as in \eqref{eqformulaofH}, $\mathcal{R}(\psi,\Phi)$ as in \eqref{eqR}, and $\mathcal{S}(\psi,\Phi)$ as \eqref{eqS} instead.
    
    As shown in Section 5.2 of \cite{ACM}, $D\Lambda$ satisfies 
    \begin{equation*}
    	    D\Lambda(\psi,\Phi)\cdot V(\psi,\Phi)(b,\xi) = (b,\xi) + Q(\psi,\Phi)\cdot\{\Lambda(\psi,\Phi),(b,\xi)\}
    \end{equation*}
    for all $(\psi,\Phi)\in U$ and $(b,\xi)\in H$, where $Q: U\subseteq F\times H\times H\rightarrow H$ is a smooth tame map that is bilinear in the last two entries. Therefore, $V$ is a right inverse of $D\Lambda$ up to a quadratic error $Q$. (Notice that $V(\psi,\Phi)$ is the exact right-inverse of $D\Lambda(\psi,\Phi)$ if $\Lambda(\psi,\Phi)=0$).
    
	In order to apply Theorem 5.1 of \cite{ACM}, it remains to check that all the maps $\Lambda$, $V$ and $Q$ are smooth tame maps on tame Fr\'echet spaces. This verification involves no further technical difficulties compared to what is done in Section 8 of \cite{ACM}, and can be done following very closely that work, while book-keeping the obvious modifications arising from the different expressions for the maps involved in the construction of $\Lambda$, $V$ and $Q$ that we computed explicitly in Sections \ref{subsectarea}, \ref{subsectfunk} and \ref{subsectnashmoserframework}. We leave the details to the meticulous reader.
		
	Thus, we obtain the desired map $\Gamma$ by applying Theorem 5.1 of \cite{ACM} to the map $\Lambda$.
\end{proof}

	In order to prove Theorem \ref{thmA}, the only thing that remains to be done is to compute $\ker D\Lambda(0,0)$.
	
\begin{prop}\label{propkernel}
	The kernel of $D\Lambda(0,0)$ is the set of all pairs $(f,\phi)\in C^{\infty}(\mathbb{S}^n)\times C^\infty_{0,odd}(T_\mathbb{S}^n)$ such that
	\begin{itemize}
		\item[$i)$] $f$ is the sum of a constant function and an odd function;
		\item[$ii)$] $\phi$ is such that, for every $v\in \mathbb{S}^n$,  
		\begin{equation}\label{eqdefphi(f)}
			\Delta_{(\Sigma_v,can)} \phi_v + (n-1)\phi_v = (n-1)\langle\nabla f,v\rangle \quad \text{on} \quad \Sigma_v.
		\end{equation}
	\end{itemize}
	Moreover, every $f$ as in $i)$ uniquely determines $\phi=\phi(f)$ as in $ii)$.
\end{prop}	

\begin{proof}
	Since the family of equators is a family of minimal hypersurfaces in $(\mathbb{S}^n,can)$, it is clear that $D_2\Lambda_1(0,0)\phi=0$ for all $\phi\in C^{\infty}_{0,odd}(T_1\mathbb{S}^n)$. On the other hand $D_1\Lambda_1(0,0)\cdot f =0$ if and only if
	\begin{equation*}
		\mathcal{F}(0,0)\cdot f = \frac{1}{vol(\mathbb{RP}^n,can)}\int_{\mathbb{RP}^n} (	\mathcal{F}(0,0)\cdot f)(\sigma)dA_{can}(\sigma).
	\end{equation*}
	By Lemma A.1 in \cite{ACM}, this is equivalent to $f=\lambda+f_0$ for some $\lambda\in \mathbb{R}$ and $f_0\in C^{\infty}_{odd}(\mathbb{S}^n)$.
	
	As for $\Lambda_2$, it follows from Lemma \ref{lemeulerlagrangeoperator} that
\begin{align*}
    \mathcal{H}_v(tf,0)&= D_1J_v(tf_v,0,t\nabla^vf_v,0)t\partial_nf_v+D_2J_v(tf_v,0,t\nabla^vf_v,0)\\
    &\quad -div_{(\Sigma_v,can)}\big(D_3J_v(tf_v,0,t\nabla^vf_v,0)\big)t\partial_nf_v \\
    &\quad -div_{(\Sigma_v,\text{can})}\big(D_4J_v(tf_v,0,t\nabla^vf_v,0)\big),
\end{align*}
so that a straightforward computation from the explicit formulas \eqref{eqJ1}, \eqref{eqJ2}, \eqref{eqJ3}, and \eqref{eqJ4} of each term $J_i$ above gives
\begin{equation*}
	(D_1\mathcal{H}(0,0)\cdot f)(x,v) =(n-1)\partial_n f_v(x)=(n-1)\inner{\nabla f(x)}{v}
\end{equation*}
for every $v\in \mathbb{S}^n$. Notice that $\inner{\nabla f(x)}{v}\in \mathbb{R}$ is an even function of $x\in \Sigma_v$, for every $v\in \mathbb{S}^n$. In particular, $C(D_1\mathcal{H}(0,0)\cdot f)\equiv 0$. 

Finally, since $D_2\mathcal{H}(0,0)=\mathcal{P}(0,0)$, we conclude that $D\Lambda_2(0,0)\cdot (f,\phi)=0$ if and only if
\begin{equation}\label{eqphif}
	\Delta_{(\Sigma_v,can)}\phi_v+(n-1)\phi_v = (n-1)\inner{\nabla f}{v}  \quad \text{on} \quad \Sigma_v,
\end{equation}
for all $v\in \mathbb{S}^n$. 
	
	Since, as in Proposition 9.1 of \cite{ACM}, for every $f\in C_{odd}^\infty(\S^n)$, there exists a unique $\phi=\phi(f)\in C^{\infty}_{0,odd}(T_1\S^n)$ solving \eqref{eqphif}, we obtain
\begin{equation*}
	\ker D\Lambda(0,0)=\{(\lambda+f_0,\phi(\lambda+f_0))\,|\,\lambda\in \mathbb{R}, f_0\in C_{odd}^\infty(\S^n)\},
\end{equation*}
as claimed.
\end{proof}

We can now finish the proof of Theorem \ref{thmA}:

\begin{proof}[Proof of Theorem \ref{thmA}] By Proposition \ref{propkernel}, given $f\in C^{\infty}_{odd}(\mathbb{S}^n)$, there exists a unique $\phi=\phi(f)\in C^{\infty}_{0,odd}(T_1\mathbb{S}^n)$ so that 
\begin{equation*}
	(f,\phi(f))\in \ker D\Lambda(0,0).
\end{equation*}
Therefore, by Theorem \ref{thmGamma}, for sufficiently small $\varepsilon>0$, the map
\begin{equation}\label{equation final solution}
    t\in (-\varepsilon,\varepsilon) \mapsto \Gamma(tf,t\phi(f)) = (\psi_t,\Phi_t) \in F
\end{equation}
describes a smooth one parameter family of pairs $(\psi_t,\Phi_t)\in \Lambda^{-1}(0,0)$ such that $(\dot{\psi}_0,\dot{\Phi}_0) = D\Gamma(0,0)\cdot (f,\phi(f))=(f,\phi(f))$. 

Thus, by Proposition \ref{propcriterion}, for each $t\in (-\varepsilon,\varepsilon)$, the family $\{\iota_{\psi_t}(\Sigma_\sigma(\Phi_t))\}$, $\sigma\in \mathbb{RP}^n$, is a Zoll family of minimal hypersurfaces in the sphere $\iota_{\psi_t}(\mathbb{S}^n)\subset \mathbb{R}^{n+1}$ with the metric induced by the embedding $\iota_{\psi_t} : \mathbb{S}^n\rightarrow \mathbb{R}^{n+1}$ as in \eqref{eqstarshapped2}. Moreover, $\psi_t=tf+o(t)$ as $t$ goes to zero.
\end{proof}

\begin{remark} \label{rmkallvariations}
	The constant functions generate the trivial deformations of $\mathbb{S}^n$ by dilations. On the other hand, any one-parameter graphical deformation $\iota_{\psi_t}(\mathbb{S}^n)$ with $\iota_{\psi_t}^*can\in \mathcal{Z}$ can be rescaled at every $t$ so that every element of the Zoll family has an area in $(\mathbb{S}^n,g_{\psi_t})$ equal to $area(\mathbb{S}^{n-1},can)$. Since the rescaling factor depends smoothly on $t$, by doing so one obtains a new smooth one-parameter family $g_{\overline{\psi}_t}\in \mathcal{Z}$ such that $\overline{\psi}_t=tf+o(t)$ for some smooth odd function $f$ on $\mathbb{S}^n$, because the first-order term in this expansion, aside from being in the kernel of $D\Lambda_1(0,0)$, is also in the kernel of $D_1\mathcal{A}(0,0)=\mathcal{F}(0,0)$. Thus, by Theorem \ref{thmA}, the smooth odd functions actually cover all possible first-order expansions of non-trivial deformations in $\mathcal{Z}$ of its kind. 
\end{remark}

\section{Symmetry}\label{Section Symmetry} 

\subsection{Equivariance}\label{subsection equivariance}
	In order to verify the symmetries of constructions based on the Nash-Moser inverse function theorem with quadratic error of R. Hamilton (\cite{HamiltonNashMoser}, Part III, Theorem 3.3.1), we need to revisit several definitions and proofs in the abstract theory of tame Fréchet spaces. We follow \cite{HamiltonNashMoser}, Part II. 

    \begin{defn}
        \normalfont Given two (graded) Fréchet spaces $X$ and $Y$, we say that $X$ is a {\it tame direct summand of} $Y$ if there exist tame linear maps $L:X\rightarrow Y$ and $M:Y\rightarrow X$ such that $ML=Id_X$. In particular, $L$ is injective, $M$ is surjective, and $Y= L(X)\oplus\ker(M)$.
    \end{defn}
	
	Let $(B,||\cdot||)$ be a Banach space. We associate to $B$ the graded Fréchet space
	\begin{equation*}
		\Sigma(B):=\{f=(f_k)_{k\in\N}\,|\,f_k\in B,\,||f||_n:=\sum_ke^{nk}||f_k||<\infty,\,\forall n\in\N\}.
	\end{equation*}

    \begin{defn}
    \normalfont A Fréchet space $X$ is {\it tame} if there exists a Banach space $B$ such that $X$ is a tame direct summand of $\Sigma(B)$, that is, if there are tame linear maps $X\xrightarrow{L}\Sigma(B)\xrightarrow{M}X$ such that $ML=Id_X$.
    In this case, we call $(B,L,M)$ the {\it tame data} of $X$.
\end{defn}
 
    \begin{exam}
        $\Sigma(B)$ is itself trivially a tame Fr\'echet space with tame data $(B,Id_{\Sigma(B)},Id_{\Sigma(B)})$.
    \end{exam}
    \begin{exam}\label{examtamefactor}
        If $X$ is a tame direct summand of a tame Fréchet space $Y$, then $X$ is naturally a tame Fréchet space. 
    To be more precise, if we have the tame summands $X\xrightarrow{L}Y\xrightarrow{M}X$ and $Y\xrightarrow{L_Y}\Sigma(B)\xrightarrow{M_Y}Y$ then $X$ is a tame Fréchet space with tame data $(B,L_YL,MM_Y)$. (\textit{Cf}. \cite{HamiltonNashMoser}, Part II, Lemma 1.3.3).
    \end{exam}
    \begin{exam}\label{examtameproduct} 
   		The Cartesian product $X_1\times X_2$ of tame Fr\'echet spaces $X_i$ with tame data $(B_i,L_i,M_i)$ is a tame Fr\'echet space with tame data $(B_1\times B_2,L_1\times L_2,M_1\times M_2)$ (\textit{cf}. \cite{HamiltonNashMoser}, Part II, Lemma 1.3.4).     
    \end{exam}

	 Although the tame data is not uniquely determined by the Fr\'echet space $X$, a choice has been fixed once and for all for each tame Fr\'echet space that has been considered in \cite{HamiltonNashMoser}. The above examples made it explicit in cases that will appear repeatedly in this work. Another important example for us is the Fr\'echet space $C^\infty(\S^n)$, which we discuss in detail below.
    
    \begin{exam}\label{exam tame data C(Sn)}
        If $X$ is a compact manifold (without boundary), then $C^{\infty}(X)$ is a tame Fr\'echet space (\cite{HamiltonNashMoser}, Part II, Theorem 1.3.6). We review the proof of this fact in the specific case of the sphere $\mathbb{S}^n$, this in order to make explicit the tame data of $C^{\infty}(\mathbb{S}^n)$.
	
	According to the proof of the aforementioned theorem, $C^{\infty}(\mathbb{S}^n)$ is a tame direct summand of the space  
	\begin{equation*}
		L_1^{\infty}(\mathbb{R}^{n+1},dx,\log(1+|x|)),
	\end{equation*}
	which consists of all measurable functions $f$ on $\mathbb{R}^{n+1}$ such that the norms $||f||_n:=\int_{\R^{n+1}}|f(x)|(1+|x|)^ndx$ are finite for all non-negative integers $n$. This space is nothing but the image by the Fourier transform of functions belonging to $C^{\infty}_0(\mathbb{R}^{n+1})$, which is the set of smooth functions with all the partial derivatives converging to zero at infinity. Moreover, this space itself is a tame direct summand of $\Sigma(L_1(\R^{n+1}))$, by \cite{HamiltonNashMoser}, Part II, Lemma 1.3.5.
	
	Let us now explicitly determine the tame data $(\hspace{-.2mm}L_1(\mathbb{R}^{n+1}),\hspace{-.1mm}L,\hspace{-.1mm}M\hspace{-.2mm})$ of $C^\infty\hspace{-.5mm}(\S^n\hspace{-.3mm})\hspace{-.5mm}$. If $\chi_k$ denotes the characteristic function of the set 
    $$\{x\in \mathbb{R}^{n+1}\,|\, k\leq \log(1+|x|)<k+1\},\quad k\in\Z_{\geq0},$$ 
    we define the maps 
    $$L_1^\infty(\R^{n+1},dx,\log(1+|x|))\xrightarrow{\tilde{L}} \Sigma(L_1(\R^{n+1}))\xrightarrow{\tilde{M}}L_1^\infty(\R^{n+1},dx,\log(1+|x|)),$$
by 
\begin{equation}\label{eqLtildeMtilde}
	\tilde{L}(f)=(f\chi_k)_k \quad \text{and} \quad \tilde{M}((f_k)_k)=\sum_k f_k\chi_k.
\end{equation} 
The tame linear maps $\tilde{L}$ and $\tilde{M}$ satisfy $\tilde{M}\tilde{L}=Id$. Hence, the Fr\'echet space $L_1^\infty(\R^{n+1},dx,\log(1+|x|))$ is tame with tame data $(L_1(\mathbb{R}^{n+1}),\tilde{L},\tilde{M})$.

	Consider the Fourier transform and its inverse, 
    $$C^\infty_0(\mathbb{R}^{n+1})\xrightarrow{\mathbb{F}} L_1^\infty(\R^{n+1},dx,\log(1+|x|))\xrightarrow{\mathbb{F}^{-1}} C^\infty_0(\R^{n+1}).$$
    Since these maps are tame linear isomorphisms, $C^\infty_0(\mathbb{R}^{n+1})$ is a tame direct summand of $L_1^\infty(\R^{n+1},dx,\log(1+|x|))$. 
	
    Let $\eta$ be an $O(n+1)$-invariant bump function that is equal to $1$ in a neighborhood of $\mathbb{S}^n$ and vanishes outside the annulus with center at the origin and radii $1/2$ and $3/2$. Let $\varepsilon:C^\infty(\S^n)\rightarrow C^\infty_0(\R^{n+1})$ be the map that first extends the function as a constant along every ray from the origin, and then multiply the result by $\eta$. Finally, denote by $\rho: C^\infty_0(\R^{n+1})\rightarrow C^\infty(\S^n)$ the restriction map. Since the maps $\rho$ and $\varepsilon$ are tame linear maps such that $\rho\circ \varepsilon = Id$, $C^\infty(\S^n)$ is a tame direct summand of $C^\infty_0(\R^{n+1})$. 
    
    In conclusion (and according to Example \ref{examtamefactor}), $C^\infty(\S^n)$ is a tame Fréchet space with tame data $(L_1(\R^{n+1}),\tilde{L}\mathbb{F}\varepsilon,\rho\mathbb{F}^{-1}\tilde{M})$.
    \end{exam}

    Since we are interested in the action of $O(n+1)$ on tame Fréchet spaces like $C^\infty(\S^n)$, we introduce a notion of compatibility.
        
\begin{defn}\label{defn compatibility of two maps AX and Ay}
    \normalfont Let $X$ be a tame direct summand of the Fréchet space $Y$, with $X\xrightarrow{L}Y\xrightarrow{M}X$ and $ML=Id_X$.
    We say that tame linear maps $A_X:X\rightarrow X$ and $A_Y:Y\rightarrow Y$ are {\it compatible} if $LA_X=A_YL$ and $A_XM=MA_Y$.
\end{defn}

\begin{defn}
    \normalfont Let $X$ be a tame Fréchet space with tame data $(B,L,M)$. 
    We say that a tame linear map $A_X:X\rightarrow X$ is {\it compatible with the tame data} if there exists a bounded linear map $A_B:B\rightarrow B$ such that $A_X$ is compatible with the induced tame map $A_B:\Sigma(B)\rightarrow\Sigma(B)$ defined naturally by $A_B((f_n)_n)=(A_Bf_n)_n$.
\end{defn}
\begin{exam}
    In any tame Fréchet space $X$, the identity and zero maps $Id,0:X\rightarrow X$ are compatible with the tame data. Indeed, simply take the identity and zero maps $ Id,0:B\rightarrow B$.
\end{exam}
\begin{exam}\label{Ay compatible -> Ax compatible}
    Let $X$ be a tame direct summand of the Fr\'echet space $Y$, and suppose that $A_X:X\rightarrow X$ and $A_Y:Y\rightarrow Y$ are compatible tame maps.
    If $Y$ is tame and $A_Y$ is compatible with the tame data of $Y$, then $A_X$ is compatible with the induced tame date of $X$ (see Example \ref{examtamefactor}).
\end{exam}
\begin{exam}\label{exam product of compatible maps}
    The product of two compatible maps is compatible with respect to the tame structure of the product (see Example \ref{examtameproduct}).
\end{exam}
	
	\begin{exam}\label{exam action of o(n+1) with C(Sn)}
        Any map $A\in O(n+1)$ induces a linear map $$T_A:C^\infty(\S^n)\rightarrow C^\infty(\S^n),$$
        simply by pre-composition with $A$.
        This map is clearly tame since it preserves all the $C^k$ norms.
        Let us verify that this map is compatible with the tame data of $C^\infty(\S^n)$ (see Example \ref{exam tame data C(Sn)}). The map $A$ determines a bounded linear map on $L_1(\mathbb{R}^{n+1})$, also by pre-composition. 
        We claim that this map satisfies the condition for $T_A$ to be compatible with the tame data of $C^{\infty}(\mathbb{S}^n)$. Indeed, this can be checked in a straightforward fashion, due to the equivariance property, under the action of $O(n+1)$, of the Fourier transform, of our choice of extension operator, and of the operators $\tilde{L}$ and $\tilde{M}$ in  \eqref{eqLtildeMtilde}.
\end{exam}

	The case of smooth functions on $T_1\mathbb{S}^n\subset \mathbb{R}^{n+1}\times \mathbb{R}^{n+1}$, which is also important for our applications, is treated similarly.

\begin{exam}\label{exam tame data C(T1Sn)}
	The tame Fr\'echet space $C^\infty(T_1\S^n)$ is such that the natural action of any $A\in O(n+1)$ (namely, pre-composition with $A\times A$) is compatible with its tame data.
	
	In fact, consider the natural inclusion $T_1\S^n\subseteq\R^{n+1}\times\R^{n+1}=\R^{2n+2}$. As before, we have tame linear maps
        $$L_1^\infty(\R^{2n+2},dx,\log(1+|x|))\xrightarrow{\hat{L}}\Sigma(L_1(\R^{2n+2}))\xrightarrow{\hat{M}}L_1^\infty(\R^{2n+2},dx,\log(1+|x|)),$$
        with $\hat{M}\hat{L}=Id$ given by the same formula \eqref{eqLtildeMtilde}. Also, we have the Fourier transform $\mathbb{F}$ and its inverse $\mathbb{F}^{-1}$ relating functions in $C^{\infty}_0(\mathbb{R}^{2n+2})$ and in $L_1^\infty(\R^{2n+2},dx,\log(1+|x|)$. Finally, we define $\hat{\rho}:C^\infty_0(\R^{2n+2})\rightarrow C^\infty(T_1\S^n)$ the restriction map, and  $\hat{\varepsilon}:C^\infty(T_1\S^n)\rightarrow C^\infty_0(\R^{2n+2})$ the map that extends a function as a constant along the normal lines of $T_1\S^n\subseteq\R^{2n+2}$, and then multiplies the resulting function by the bump function $\hat{\rho}$ given by $\hat{\rho}(x,y)=\rho(x)\rho(y)$ for every $x,y\in\R^{n+1}$, for $\rho$ as in Example \ref{exam tame data C(Sn)} (the support of $\rho$ must be taken smaller in order that the tubular neighborhood theorem applies).
    
    It is straightforward to check that $C^\infty(T_1\S^n)$ is a tame Fréchet space with tame date $(L_1(\R^{2n+2}),\hat{L}\mathbb{F}\hat{\varepsilon},\hat{\rho}\mathbb{F}^{-1}\hat{M})$.
    
    Observe that every $A\in O(n+1)$ determines a map
        $$T_A:C^\infty(T_1\S^n)\rightarrow C^\infty(T_1\S^n),$$
        $$(T_Af)(x,y):=f(Ax,Ay).$$
    As in the case of $\S^n$, the natural map $T_A : L_1(\R^{2n+2})\rightarrow L_1(\R^{2n+2})$, induced by pre-composition with $(A\times A)(x,y)=(Ax,Ay)$, allows us to verify that $T_A$ is compatible with the tame data of $C^\infty(T_1\S^n)$.
    \end{exam}
	
	The proof of Theorem 5.1 in \cite{ACM}, on which our main result Theorem \ref{thmA} relies, is based on the concept of {\it near-projection} introduced in \cite{HamiltonDeformationComplex}. This is a smooth tame map $G:U\subseteq X\rightarrow X$, defined on an open subset $U\subseteq X$ of a tame Fréchet space $X$, that satisfies
\begin{equation}\label{def near-projection}
    (DG)_x(G(x)-x)+Q(x)\{G(x)-x,G(x)-x\}=0 \quad\forall x\in U,
\end{equation}
for some smooth tame map $Q=Q(\cdot)\{\cdot,\cdot\}:(U\subset X)\times X\times X\rightarrow X$ that is linear in the last two coordinates. We call $Q$ the \textit{associated quadratic form}. 

	As shown in \cite{HamiltonDeformationComplex}, (see the Theorem on page 26), the set of fixed points $Fix(G)$ of such map $G$ is locally the set of fixed points of an actual smooth tame projection $P$, that is, a tame map satisfying $P^2=P$. We call it the \textit{realizing projection}. As will become clear from the construction, this projection is not unique, it may depend on the tame data and on the choice of \textit{smoothing operators}.

	The next theorem shows that the specific realizing projection constructed by Hamilton satisfies an additional property. Namely, if a tame linear map $T$ is compatible with the tame data of $X$ and commutes with the near-projection $G$, then there exists a realizing projection that also commutes with $T$.

\begin{thm}\label{thmequivarianceHam2}
	Let $X$ be a tame Fr\'echet space and $G:U\subseteq X\rightarrow X$ be a near-projection defined on an open neighborhood of $0$ in $X$.
    
    Then there exists an open neighborhood $W\subset U$ of $Fix(G)$ and a smooth tame map $P:W\subseteq X\rightarrow W\subseteq X$ with the following properties: $Fix(P)=Fix(G)\cap W$, $PP=P$, $P(0)=0$, and $DP(0)u=u$ whenever $DG(0)u=u$.
    
    Moreover, if $T: X\rightarrow X$ is a tame linear map that is compatible with the tame data of $X$ and $GT=TG$ on $T^{-1}(U)\cap U$ then
	\begin{equation*}
	    PT=TP \quad \text{on} \quad T^{-1}(W)\cap W.
	\end{equation*}
\end{thm}

\begin{proof}
    We review the proof of Hamilton's Theorem on page 26 of \cite{HamiltonDeformationComplex}, with focus on checking the equivariance of $P$ with respect to the compatible tame linear map $T: X\rightarrow X$ that satisfies $GT=TG$ on $T^{-1}(U)\cap U$.
    
    Let $(B,L,M)$ be the tame data of $X$ and let $\overline{T}: B\rightarrow B$ be the bounded linear map such that $LT=\overline{T}L$ and $M\overline{T}=TM$. Recall that $L$ is injective, $M$ is surjective, and $\Sigma(B)=L(X)\oplus\ker(M)$.
    
    We set $\overline{U}:=L(U)\oplus\ker(M)\subseteq\Sigma(B)$ and define the natural extension of $G$, that is, the map $\overline{G}:\overline{U}\subseteq \Sigma(B)\rightarrow\Sigma(B)$ given by
    $$\overline{G}(L(x)+k):=LG(x),\quad\forall x\in U,\,\forall k\in\ker(M).$$
    By construction, $\overline{G}L=LG$ and $M\overline{G}=GM$. Moreover, $\overline{L}x+k\in \overline{U}$ belongs to $Fix(\overline{G})$ if and only if $k=0$ and $LGx=Lx$, which is equivalent to $k=0$ and $Gx=x$, because $L$ is injective. In other words, $Fix(\overline{G})=L(Fix(G))$. 
    
	We claim that the map $\overline{G}$ is a near-projection on the open set $\overline{U}=M^{-1}(U)$, with $\overline{G}(0)=0$. In fact,  since $G$ satisfies (\ref{def near-projection}) for some associated quadratic form $Q:U\times X\times X\rightarrow X$, we can define the natural extension of $Q$ to $\Sigma(B)$, that is, the map $\overline{Q}:\overline{U}\times\Sigma(B)\times\Sigma(B)\rightarrow\Sigma(B)$, linear in the last two variables, given by
    $$\overline{Q}(L(x_0)+k_0)\{L(x_1)+k_1,L(x_2)+k_2\}=L(Q(x_0)\{x_1,x_2\}),$$
    for all $x_0\in U$, $x_1,x_2\in X$ and $k_0,k_1,k_2\in\ker(M)$. Since, for $x\in U$ and $k\in\ker(M)$ we have $\overline{G}(L(x)+k)-(L(x)-k)=L(G(x)-x)-k$, it follows that
    $$(D\overline{G})_{L(x)+k}(\overline{G}(L(x)+k)-L(x)-k)=L\big((DG)_x(G(x)-x)\big),$$
    and
    \begin{multline*}
        \overline{Q}(L(x)+k)\{\overline{G}(L(x)+k)-L(x)-k,\overline{G}(L(x)+k)-L(x)-k\} \\
        =L(Q(x)\{G(x)-x,G(x)-x\}).
    \end{multline*}
    Adding up the two previous equations, we conclude that $\overline{G}$ is a near-projection with $\overline{Q}$ as its associated quadratic form. By construction, both $\overline{G}$ and $\overline{Q}$ are smooth tame maps, and the claim is proven.
    
	We also have 
	\begin{equation*}
		\overline{G}\,\overline{T}=\overline{T}\,\overline{G} \quad \text{on} \quad \overline{T}^{-1}(\overline{U})\cap \overline{U}.
	\end{equation*} 
	In fact, for all $k\in ker(M)$, $MTk=\overline{T}Mk=0$, while for all $x\in X$, $\overline{T}Lx=LTx$. Hence, for every $x\in U$ and $k\in ker(M)$, $L(x)+k\in \overline{T}^{-1}(\overline{U})\cap \overline{U}$ and
    \begin{align*}
        \overline{G}(\overline{T}(L(x)+k))&=\overline{G}(L(Tx)+Tk)=LG(Tx)  \\
        & =LTG(x)=\overline{T}LG(x)= \overline{T}\,\overline{G}(L(x)+k).
    \end{align*}
    
    Let $s:\R\rightarrow\R$ be a smooth map satisfying $s(t)=0$ for $t\leq0$, $s(t)=1$ for $t\geq1$, and $0\leq s(t)\leq1$ in between. 
    For every $t>0$, the {\it smoothing operator} $S_t:\Sigma(B)\rightarrow\Sigma(B)$ maps the sequence $f=(f_k)$ in $\Sigma(B)$ to the sequence $S_t(f)=(S_t(f)_k)$ where
    $$S_t(f)_k:=s(\log(t)-k)f_k.$$
   (\textit{Cf}. \cite{HamiltonNashMoser}, Part III, Section 1.4). Notice that $S_t(f_k)\rightarrow (f_k)$ as $t\rightarrow \infty$. Also, they satisfy the estimates on page 28 of \cite{HamiltonDeformationComplex}, for every $t\geq 1$ and positive integer $\ell\leq m$, 
   \begin{align*}
  	||S_t((f_k))||_m\leq Ct^{m-\ell+s}||(f_k)||_{\ell} , \\ ||(I-S_t)((f_k))||_\ell\leq Ct^{\ell-m+s}||(f_k)||_{m},
  \end{align*}
  for constants that $C$ may depend on $m$ and $n$ and an integer $s\geq 2$ that depends only on the degree of the tame estimates of $\overline{G}$ and $\overline{Q}$ (compare with Lemma 1.4.1 in \cite{HamiltonNashMoser}, Part III, noticing that we took $t\geq 1$).
  
  Finally, it is clear from the definition that $\overline{T}S_t=S_t\overline{T}$ for every $t$. 
    
    Fix a real number $t_0\geq 3$. Given $\overline{x}_0\in\overline{U}\subseteq\Sigma(B)$, consider the algorithm
    \begin{align*}
        t_{n+1}&:=t_n^{3/2},\\
        \overline{x}_{n+1}&:=(I-S_{t_n})\overline{x}_n+S_{t_n}\overline{G}(\overline{x}_n).
    \end{align*}
    Notice that $t_n$ is increasing and divergent, and that $\overline{x}_n$ is the constant sequence if $\overline{x}_0\in Fix(\overline{G})$.
    As proved in \cite{HamiltonDeformationComplex}, pages 26-40, there exists an open subset $\overline{W}\subseteq \overline{U}$ with $Fix(\overline{G})\subseteq\overline{W}$, such that, for every $\overline{x}_0\in\overline{W}$, 
    $$\overline{P}(\overline{x}_0):=\lim_{n\to\infty} \overline{x}_n,$$
    is well defined and determines a smooth tame map $\overline{P}:\overline{W}\subseteq X\rightarrow \overline{W}\subseteq X$ such that 
    $\overline{P}\,\overline{P}=\overline{P}$ and $Fix(\overline{P})=Fix(\overline{G})\cap \overline{W}$. Since the algorithm, for every $\overline{x}_0=Lx_0+k\in \overline{U}$, defines the same sequence $\overline{x}_n$, $n\geq 1$, for every $k\in Ker(M)$, we may take $\overline{W}$ of the form $L(W)\oplus Ker(M)$ for the open subset $W=L^{-1}(\overline{W})$.
    
    We claim that $\overline{P}\,\overline{T}=\overline{T}\,\overline{P}$ on $\overline{T}^{-1}(\overline{W})\cap \overline{W}$. In fact, suppose that $\overline{x}_0\in \overline{U}$ is such that $\overline{y}_0:=\overline{T}(\overline{x}_0)\in \overline{U}$. Since $\overline{G}\,\overline{T}=\overline{T}G$ and $S_t\overline{T}=\overline{T}S_t$ for every $t$, it is straightforward to check that the sequence $\{\overline{y}_n\}$ that the algorithm defines for the initial condition $\overline{y}_0$ is the sequence $\overline{y}_{n}=\overline{T}(\overline{x}_n)$, where $\{\overline{x}_n\}$ is the sequence the algorithm defines for the initial condition $\overline{x}_0$. Thus, by definition of $\overline{P}$ and the continuity of $\overline{T}$, whenever $\overline{x}_0\in \overline{W}\cap \overline{T}^{-1}(\overline{W})$, we have $\overline{P}\,\overline{T}(x_0) = \overline{P}\overline{y}_0=\lim \overline{y}_n=\lim\overline{T}(\overline{x}_n)=\overline{T}(\lim \overline{x}_n)=\overline{T}\,\overline{P}(x_0)$.
    
    The open subset $W:=L^{-1}(\overline{W})=  M(\overline{W})= U$ contains $0$. Define $P:=M\overline{P}L:W\subseteq X\rightarrow W\subseteq X$, which is a smooth tame map (as a composition of smooth and tame maps).
   
   	In order to prove that $PP=P$ and $Fix(P)=Fix(G)\cap W$, it is enough to check that $Fix(G)\cap W=Fix(P)=Im(P)$. 
   	
   	By the definition of $\overline{G}$, we deduce that $Fix(\overline{G})\cap \overline{W}=L(Fix(G)\cap W)$. Since $Fix(\overline{G})\cap \overline{W}=Fix(\overline{P})=Im(\overline{P})$, we conclude in particular that $Im(\overline{P})= L(Fix(G)\cap W)$. Also, if $x\in W$, then $Lx\in \overline{W}$ and
    \begin{equation*}
		G(x)=x\, \Rightarrow \overline{G}(Lx)=Lx\, \Rightarrow  \overline{P}(Lx)=Lx\, \Rightarrow Px=M\overline{P}(Lx)=MLx=x,
    \end{equation*}
    so that $Fix(G)\cap W\subseteq Fix(P)$. Finally,
    \begin{multline*}
    	Fix(P)\subseteq Im(P)=M\overline{P}L(W) \subseteq M(Im(\overline{P})) \\ = ML(Fix(G)\cap W)=Fix(G)\cap W. 
    \end{multline*}
    Thus $Fix(P)=Im(P)=Fix(G)\cap W$, as claimed.
    
    We have $G(0)=0$, and the algorithm starting at $\overline{x}_0=0$ clearly converges to $0$. Thus $P(0)=M\overline{P}L(0)=0$. Similarly, as in \cite{HamiltonDeformationComplex} pages 37-38, and using $ML=Id$, we can check that
    \begin{equation*}
    	DG(0)\cdot u=u\, \Rightarrow D\overline{G}(0)\cdot Lu=Lu\, \Rightarrow D\overline{P}(0)\cdot Lu=Lu\, \Rightarrow DP(0)\cdot u=u.
    \end{equation*}
    Finally, the compatibility of $T$ implies that, 
    \begin{equation*}
   		PT=M\overline{P}LT=M\overline{P}\,\overline{T}L= M\overline{T}\,\overline{P}L=TM\overline{P}L=TP
   	\end{equation*}
on $T^{-1}(W)\cap W$, as we wanted to show.
\end{proof}

	Using Theorem \ref{thmequivarianceHam2}, we prove the equivariant version of Theorem 3.3.1. in Part III of \cite{HamiltonNashMoser} (compare Theorem 5.1 in \cite{ACM}).

\begin{thm}\label{thmequivariantHamIFTquadraticerror}
Let $F$ and $H$ be tame Fr\'{e}chet spaces and let $\Lambda$ be a smooth tame map defined on an open set $U\subseteq F$ containing the origin,
\begin{equation*}
\Lambda:U \subseteq F \rightarrow H,
\end{equation*}
with $\Lambda(0)=0$. Suppose there exists a smooth tame map $V(f)h$ linear in $h$,
\begin{equation*}
V:U\times H\rightarrow F,
\end{equation*}
and a smooth tame map $Q(f)\{h,k\}$ bilinear in $h$ and $k$,
\begin{equation*}
Q:U\times H \times H \rightarrow H,
\end{equation*}
such that, for all $f\in U$ and all $h\in H$,
\begin{equation*}
D\Lambda(f)V(f)h = h+Q(f)\{\Lambda(f),h\}.
\end{equation*}
	Then there exists a neighborhood $W\subseteq U$ with $0\in W$, and a smooth tame map
\begin{equation*}
\Gamma:Ker D\Lambda(0)\cap W \rightarrow \Lambda^{-1}(0),
\end{equation*}
such that
\begin{equation*}
	\Gamma(0)=0 \quad \text{and} \quad D\Gamma(0)v=v \quad \text{for every} \quad v\in Ker(D\Lambda(0)).
\end{equation*}
	
	If, moreover, there exist tame linear maps $T_F: F\rightarrow F$ and $T_H : H \rightarrow H$ that are compatible with the respective tame data and are such that
\begin{equation*}
	\Lambda T_F=T_H\Lambda \quad \text{on} \quad T_F^{-1}(W)\cap W 
\end{equation*}
and
\begin{equation*}
	V(T_F\times T_H)=T_FV \quad \text{on} \quad (T_F\times T_H)^{-1}(W\times H)\cap W\times H.
\end{equation*}
then the map $\Gamma$ also satisfies
\begin{equation*}
	\Gamma T_F = T_F\Gamma \quad \text{on} \quad Ker D\Lambda(0)\cap T_F^{-1}(W)\cap W.
\end{equation*}
\end{thm}

\begin{proof}
	According to the proof of the aforementioned Theorem 3.3.1 in \cite{HamiltonNashMoser} (see \cite{ACM}, Theorem 5.1), the smooth tame map
	\begin{equation*}
		G : (U\subset F) \times H \rightarrow F\times H
	\end{equation*}
	given by
	\begin{equation*}
		G(f,h)= (f - V(f)\Lambda(f),h-\Lambda(f))
	\end{equation*}
	is a near-projection, so that, for every $x=(f,g)\in U\times H$,
	\begin{equation*}
		DG(x)\cdot (G(x) - x) + \widetilde{Q}(x)\cdot \{G(x)-x,G(x)-x\} = 0 
	\end{equation*}
	for the smooth tame map 
	\begin{equation*}
		 \widetilde{Q} : (U\subset F) \times H \times (F\times H) \times (F\times H) \rightarrow F\times H
	\end{equation*}
	given by
	\begin{multline*}
		\widetilde{Q}(f,h)\cdot\{(\widetilde{a},\widetilde{c}),(a,c)\} {=} \\
		-\left([DV(f)\cdot\{\widetilde{a}\}]\cdot \{c\} + V(f)Q(f)\cdot\{\widetilde{c},c\}, Q(f)\cdot\{ \widetilde{c},c \} \right).
	\end{multline*}
	It is then straightforward to check that the compatible maps $T_F$ and $T_H$ are such that
	\begin{equation*}
		G(T_F\times T_H)=(T_F\times T_H)G \quad \text{on} \quad (T_F\times T_H)^{-1}(U\times H)\cap U\times H.
	\end{equation*}
	
	Notice that, also by construction, $G(f,h)=(f,h)$ if and only if $\Lambda(f)=0$. In other words, 
	\begin{equation*}
		\Lambda^{-1}(0)\cap U = \pi_F(Fix(G)),
	\end{equation*}
	where $\pi_F : F\times H \rightarrow F$ is the standard projection. 
	
	By Theorem \ref{thmequivarianceHam2}, there exists a smooth tame map
	\begin{equation*}
		P : \widetilde{W} \subset (F\times H) \rightarrow \widetilde{W}\subseteq  (F\times H),
	\end{equation*}		
	defined on an open neighborhood of the origin ${\widetilde{W}}\subset U\times H$, that is a realizing projection $P=PP$ with $P(0)=0$, $DP(0)\cdot u=u$ whenever $DG(0)\cdot u = u$, and
	\begin{equation*}
		P(T_F\times T_H)=(T_F\times T_H)P \quad \text{on} \quad \widetilde{W}\cap (T_F\times T_H)^{-1}(\widetilde{W}).
	\end{equation*}
	
	Finally, letting $W=\pi_F(\widetilde{W}) \subset F$ and  
	\begin{equation*}
		\Gamma : Ker(D\Lambda(0))\cap W \rightarrow F
	\end{equation*}
	be defined by
	\begin{equation*}
		\Gamma(f) = \pi_F (P(f,0)),
	\end{equation*}		
	we obtain the map $\Gamma$ with all the desired properties including the equivariance, because, for every $f\in Ker(D\Lambda(0))\cap W\cap T_F^{-1}(W)$,
	\begin{multline*}
		\Gamma(T_F(f))=\pi_F(P(T_Ff,0))=\pi_F(P(T_F\times T_H)(f,0))\\=\pi_F((T_F\times T_H)P(f,0))=T_F\pi_F(P(f,0))=T_F\Gamma(f).
	\end{multline*}
\end{proof}

\subsection{Applications} \label{subsectionapplication}
The main ingredient in the proof of Theorem \ref{thmA} was Theorem 5.1 of \cite{ACM} (which is Theorem 3.2.1. in Part III of \cite{HamiltonNashMoser}). In the last section, we extended that result, with the help of the notion of compatible maps, so that we can now verify equivariance with respect to suitable tame maps (see Theorem \ref{thmequivariantHamIFTquadraticerror}). This has laid down the path towards the proofs of Theorems \ref{thmB}, \ref{thmC}, \ref{thmD}, and \ref{thmE}.

	In our setup, the key step is to check the compatibility and equivariance with respect to the natural action of $O(n+1)$.

	Let $\Lambda: U\subseteq F\rightarrow H$ be the smooth tame map defined in (\ref{Lambda1 and Lambda2}), defined on an open neighborhood of $0$ in the tame Fréchet space $F$ and taking values on the tame Fr\'echet space $H$.

\begin{lem}\label{O(n+1) action is compatible with tame data}
    We can choose the tame data for the Fréchet spaces $F$ and $H$ in \eqref{F and H} in such a way that the natural action of every $A\in O(n+1)$ is compatible with their tame data.
\end{lem}
\begin{proof}
	According to Examples \ref{exam action of o(n+1) with C(Sn)} and \ref{exam tame data C(T1Sn)}, the action of every $A\in O(n+1)$ is compatible with the chosen tame data of $C^{\infty}(\mathbb{S}^n)$ and $C^{\infty}(T_1\mathbb{S}^n)$. We will show that the spaces  $C_0^\infty(\mathbb{RP}^n)$ and $C_{0,\text{odd}}^\infty(T_1\S^n)$ are tame direct summands of $C^{\infty}(\mathbb{S}^n)$ and $C^\infty(T_1\S^n)$, respectively. Choosing their tame data as in Example \ref{examtamefactor}, and the tame data of the Cartesian products $F$ and $H$ as in Example \ref{examtameproduct}, the compatibility with the natural action of every $A\in O(n+1)$ will follow immediately (see Examples \ref{Ay compatible -> Ax compatible} and \ref{exam product of compatible maps}).

    Let $i_0:C^\infty_0(\RP^n)\rightarrow C^\infty(\RP^n)$ be the inclusion and $p_0:C^\infty(\RP^n)\rightarrow C^\infty_0(\RP^n)$ be the tame linear map given by 
    $$p_0(f)=f-\frac{1}{\text{Vol}(\RP^n)}\int_{\RP^n}f(\sigma)dA_{\text{can}}(\sigma).$$
    Since $p_0i_0=Id$, $C^\infty_0(\RP^n)$ is a tame direct summand of $C^\infty(\RP^n)$.

    On the other hand, let $\pi^*:C^\infty(\RP^n)\rightarrow C^\infty(\S^n)$ be the pull-back of the canonical projection $\pi:\S^n\rightarrow\RP^n$, and $p_1:C^\infty(\S^n)\rightarrow C^\infty(\RP^n)$ the tame linear map given by
    $$\big(p_1(f)\big)([v])=\frac{1}{2}(f(v)+f(-v)),\quad\forall [v]\in\RP^n.$$
    Since $p_1\pi^*=Id$, $C^\infty(\RP^n)$ is a tame direct summand of $C^\infty(\S^n)$. 
    
    Hence, $C_0^\infty(\RP)$ is a tame direct summand of $C^\infty(\S^n)$ and has induced tame data as in Example \ref{examtamefactor}. Moreover, the natural action of every $A\in O(n+1)$ in each space $C_0^\infty(\RP)$ and $C^\infty(\S^n)$ is compatible in the sense of Definition \ref{defn compatibility of two maps AX and Ay}. Therefore, this action is compatible with the induced tame data of $C^\infty_0(\RP^n)$ (by Example \ref{Ay compatible -> Ax compatible}).

    The space $C^{\infty}_{0,odd}(T_1\mathbb{S}^n)$ is treated similarly. Let $k:C^\infty_{0,odd}(T_1\S^n)\rightarrow C^\infty(T_1\S^n)$ be the inclusion map, and consider the map $q:C^\infty(T_1\S^n)\rightarrow C^\infty_{0,odd}(T_1\S^n)$ given by the composition of the tame linear maps 
    \begin{align*}
        C^\infty(T_1\S^n)&\rightarrow C^\infty_{*,odd}(T_1\S^n)\\
        \Phi&\mapsto\frac{1}{2}(\Phi(\cdot,\cdot)-\Phi(\cdot,-\cdot)),
    \end{align*}
    and (recall \eqref{eqdecompositionC*odd})
    \begin{align*}
        C^\infty_{*,odd}(T_1\S^n)&\rightarrow C^\infty_{0,odd}(T_1\S^n)\\
        \Phi&\mapsto\Phi-jC\Phi.
    \end{align*}
    Since that $qk=Id$, $C^\infty_{0,odd}(T_1\S^n)$ is a tame direct summand of $C^\infty(T_1\S^n)$. The natural action on both spaces of each $A\in O(n+1)$ is compatible. Therefore, it is also compatible with the induced tame data of $C^\infty_{0,\text{odd}}(T_1\S^n)$ (see Example \ref{Ay compatible -> Ax compatible}).
\end{proof}

	This tame data for $F$ and $H$ is fixed once and for all.

\begin{lem}\label{lemmequivariantLambdaV}
   The tame maps $\Lambda$ and $V$ in Theorem \ref{thmGamma} are equivariant with respect to the natural action of any $A\in O(n+1)$.
\end{lem}
\begin{proof}
    A straightforward computation shows the compatibility of the natural action of any $A\in O(n+1)$ with all maps involved in the definition of $\Lambda$ and $V$. Thus, $\Lambda A=A\Lambda$ and $VA=AV$ on $U$ (which can be itself chosen to be invariant under every action by an element of $O(n+1)$, since the action preserves $C^k$ norms of functions in $C^{\infty}(\mathbb{S}^n)$ and $C^{\infty}(T_1\mathbb{S}^n)$ and we may take $U$ as a sufficiently small ball around the origin in the $C^{3n+4}$ norm).
\end{proof}

\begin{prop}\label{propparathmB}
    In the context of Theorem \ref{thmA} and its proof, take $A\in O(n+1)$ and consider the maps $f$ and $ \overline{f}=fA\in C^{\infty}_{odd}(\mathbb{S}^n)$. Then
	\begin{equation*}
		A\overline{\iota}_t=\iota_tA,\quad A^*g_t=\overline{g}_t\quad \text{and} \quad A\Sigma_v(\overline{\Phi}_t)=\Sigma_{Av}(\Phi_t),
	\end{equation*}
\end{prop}
\begin{proof}
    Given $f\in C^{\infty}_{odd}(\mathbb{S}^n)$, let $\phi=\phi(f)\in C^\infty_{0,\text{odd}}(T_1\S^n)$ be the unique function such that  $(f,\phi)\in\ker(D\Lambda)(0,0)$ (see Proposition \ref{propkernel}). That is, for each $v\in \mathbb{S}^n$, $\phi(f)_v$ is the unique solution of \eqref{eqdefphi(f)} that is $L^2$-orthogonal to the linear functions of $\Sigma_v$. For every $A\in O(n+1)$, it is straightforward to check that, given $\overline{f}=f A$, then $\overline{\phi}:=\phi A$ is such that $(\overline{f},\overline{\phi})\in\ker(D\Lambda)(0,0)$, that is, $\overline{\phi}=\phi(\overline{f})$. 
    
     By virtue of Lemma \ref{O(n+1) action is compatible with tame data} and Lemma \ref{lemmequivariantLambdaV}, the abstract Theorem \ref{thmequivariantHamIFTquadraticerror} can be applied to the maps $\Lambda$ and $V$. Thus, the map $\Gamma$ is equivariant with respect to the natural action of every $A\in O(n+1)$.
     
    This means that, if $(\psi_t,\Phi_t)=\Gamma(t(f,\phi(f)))$ are the maps that determine the metrics $g_t\in \mathcal{Z}$ and the Zoll families $\{\Sigma_\sigma(\Phi_t)\}$ associated to the initial data $(f,\phi)$, then 
    \begin{multline*}
    	(\overline{\psi}_t,\overline{\Phi}_t):=\Gamma(t(\overline{f},\overline{\phi}))=\Gamma(t(f,\phi)\circ A) \\
    =\Gamma(t(f,\phi))\circ A=(\psi_t,\Phi_t)\circ A=(\psi_t A,\Phi_t A).
   \end{multline*}
   
	The equality $\overline{\psi}_t=\psi_tA$ has an intuitive geometric interpretation. If we pre-compose the initial condition $f$ by $A\in O(n+1)$, then the embedding $\overline{\iota}_t$ constructed in Theorem \ref{thmA} for the initial condition $\overline{f}=fA$ is related to the embedding $\iota_t$ associated with the initial condition $f$ by $A\overline{\iota}_t =\iota_t A$. It follows that the metrics induced on $\mathbb{S}^n$ are related by $\overline{g}_{t}=A^*g_t$.
	
	The equality $\overline{\Phi}_t=\Phi_t A$ also has an intuitive meaning. In fact,  
	\begin{align*}
		A(\Sigma_{v}(\overline{\Phi}_t)) & = \{ \cos(\overline{\Phi}_t(x,v)Ax+\sin(\overline{\Phi}_t(x,v))Av\,|\,x\in \Sigma_v\} \\
		& = \{ \cos(\Phi_t(Ax,Av)Ax+\sin(\Phi_t(Ax,Av))Av\,|\,x\in \Sigma_v\} = \Sigma_{Av}(\Phi_t).
	\end{align*}
\end{proof}

	An immediate consequence is the following:

\begin{cor}\label{corfirstpartthmB}
	Under the conditions of Proposition \ref{propparathmB}, every element of
$$G_f=\{A\in O(n+1)\,|\,fA=f\}$$ 
maps $\iota_t(\mathbb{S}^n)\subset \mathbb{R}^{n+1}$ into itself, and maps members of the Zoll family of this hypersurface into members of the Zoll family. In particular, $G_f$ can be identified with a closed Lie subgroup of the isometry group of $\iota_t(\mathbb{S}^n)$.
\end{cor}

\begin{proof}
	According to Proposition \ref{propparathmB}, in the particular case where $\overline{f}=f$, that is, if $A\in G_f$, we have
	\begin{equation*}
		A\iota_t=\iota_tA,\quad A^*g_t=g_t\quad \text{and} \quad A\iota_t(\Sigma_v(\Phi_t))=\iota_t(A\Sigma_v(\Phi_t))=\iota_t(\Sigma_{Av}(\Phi_t)).
	\end{equation*}
	The result follows.
\end{proof}

\begin{remark}\label{rmkconformalequivariance}
	As the reader will notice, the same arguments, with minor notational changes, shows that Theorem A in \cite{ACM} also has an equivariant counterpart, and therefore a statement analogous to Corollary \ref{corfirstpartthmB} holds. We will discuss this further in Appendix \ref{appequivarainceconformal}.
\end{remark}

	To finish this section, let us give another application of Theorem \ref{thmequivariantHamIFTquadraticerror}, in the context of Theorem E of \cite{ACM}, and prove Theorem \ref{thmE}.

	Given a closed manifold $M$, denote by $Sym_2(M)\rightarrow M$ the vector bundle of symmetric two-tensors on $M$. Also, if $E\rightarrow M$ is a vector bundle, we denote by $\Gamma(E)$ the set of its smooth sections. According to \cite{HamiltonNashMoser}, Part II, Corollary 1.3.9, $\Gamma(E)$ is a tame Fr\'echet space. Let us make explicit the tame data of this space for $\Gamma(Sym_2(\S^n))$.
    \begin{lem}\label{lem O(n+1) compatible with Sym2(Sn)}
        We can choose tame data for $\Gamma(Sym_2(\S^n))$ in such a way that the action of every $A\in O(n+1)$ via pull-back is compatible with the tame data. 
    \end{lem}
    \begin{proof}
    Let $S_2$ be the set of bilinear and symmetric maps on $\R^{n+1}$, $S^2$ is a vector space of dimension $(n+1)(n+2)/2$. Observe that if $(B,L,M)$ is the tame data of $C^\infty(\S^n)$ (see Example \ref{exam tame data C(Sn)}), then the tensor product $C^\infty(\S^2)\otimes S_2$ has tame data $(B\otimes S_2,L\otimes Id,M\otimes Id)$.
    
    Recall that the pre-composition by $A\in O(n+1)$, say $T_A:C^\infty(\S^n)\rightarrow C^\infty(\S^n)$, is compatible with the tame data of $C^\infty(\S^n)$ (see Example \ref{exam action of o(n+1) with C(Sn)}), so let $T:B\rightarrow B$ be the map that induces the compatibility.
    A straightforward computation shows that $T_A\otimes A^*:C^\infty(\S^n)\otimes S_2\rightarrow C^\infty(\S^n)\otimes S_2$ is compatible with the tame data by considering the operator $T\otimes A^*:B\otimes S_2\rightarrow B\otimes S_2$.
    Observe that if $\map{i}{\S}{n}{\R^{n+1}}$ is the natural inclusion, then $C^\infty(\S^n)\otimes S_2$ is naturally identifiable with $\Gamma(i^*(Sym_2(\R^{n+1})))$, and the map $T_A\otimes A^*$ is simply the action of $A^*$ in this space.

    If $N$ is the outward-pointing unit normal vector of $S^n$ then it is not difficult to show that
    \begin{align*}
        \Gamma(i^*(Sym_2(\R^{n+1})))&\rightarrow\Gamma(Sym_2(\S^n))\oplus\Gamma(i^*(\Omega^1(\R^{n+1})))\\
        T&\mapsto\big(T|_{T\S^n\times T\S^n}, T(N,\cdot)\big),
    \end{align*}
    is a tame isomorphism of tame Fr\'echet spaces.
    This shows immediately that $\Gamma(Sym_2(\S^n))$ is a tame direct summand of $\Gamma(i^*(Sym_2(\R^{n+1}))$.
    Moreover, it is straightforward to verify that the action of $A$ (by pull-back) on each space is compatible (in the sense of Definition \ref{defn compatibility of two maps AX and Ay}). 
    Examples \ref{examtamefactor} and \ref{Ay compatible -> Ax compatible} give the desired tame data for $\Gamma(Sym_2(\S^n))$.
    \end{proof}

	The maps $\Lambda$ and $V$ described in Section 9.5 of \cite{ACM} are equivariant with respect to the action of every $A\in O(n+1)$, as can be checked in a straightforward fashion. Thus, applying Theorem \ref{thmequivariantHamIFTquadraticerror}, we obtain similar corollaries about the metrics in $\mathcal{Z}$ constructed by the proof of Theorem E in \cite{ACM}. In particular, if the initial condition in Theorem $E$ in \cite{ACM} is a smooth divergence-free symmetric two-tensor $h$ that has a stabilizer group $G\subseteq O(n+1)$, then the one-parameter metrics induced by this initial condition are fixed by the action of $G$, and the members of the corresponding Zoll families of minimal spheres are permuted by the action of $G$.

\begin{proof}[Proof of Theorem \ref{thmE}]
	Let $A:=-Id\in O(n+1)$. Every smooth diver\-gen\-ce-free, traceless symmetric two-tensor $h$ on $(\mathbb{RP}^n,can)$ lifts to a smooth di\-ver\-gen\-ce-free, traceless symmetric two-tensor $\overline{h}$ on $(\mathbb{S}^n,can)$ such that $A^*\overline{h}=\overline{h}$. By the equivariance of the map $\Gamma$ constructed in the proof of Theorem E of \cite{ACM},  we conclude that the metrics $\overline{g}_t\in \mathcal{Z}$ constructed therein satisfy $A^*\overline{g}_t=\overline{g}_t$ (as long as we use the tame data described in Lemma \ref{lem O(n+1) compatible with Sym2(Sn)} for $\Gamma(Sym_2(\S^n))$). Moreover, similar to how it was argued in Proposition \ref{propparathmB}, for every member $\overline{\Sigma}^t_{[v]}$, $[v]\in \mathbb{RP}^n$, of the Zoll family of $(\mathbb{S}^n,\overline{g}_t)$, one can check that 	
	\begin{equation*}
		A(\overline{\Sigma}^t_{[v]})=\overline{\Sigma}^{t}_{[Av]}=\overline{\Sigma}^{t}_{[-v]}=\overline{\Sigma}^t_{[v]}.
	\end{equation*} 
	Passing to the quotient, we obtain a smooth one-parameter family of metrics $g_t$ on $\mathbb{RP}^n$, with $\dot{g}_t=can+ht+o(t)$ as $t$ goes to zero, and, for each parameter $t$, a Zoll family of embedded minimal real projective hyperplanes $\Sigma^t_\sigma:=\overline{\Sigma}_\sigma^t/\{\pm Id\}$ in $(\mathbb{RP}^n,g_t)$, parametrized by $\mathbb{RP}^n$.
\end{proof}

\begin{remark}\label{rmkkilling}
	Let $\mathcal{K}\subset \mathcal{Z}$ be the space of metrics on $\mathbb{RP}^n$ with minimal linear real projective hyperplanes, and $\mathcal{K}_0\subset\mathcal{K}$ be the subset consisting of those metrics with the same volume as $can$. Let $\mathcal{K}_2$ be the space of Killing symmetric two-tensors in $(\mathbb{RP}^n,can)$, and $\mathcal{K}_{2,n}\subset \mathcal{K}_2$ be the subset consisting of those with $tr_{can}k=n$. A symmetric two-tensor $k$ belongs to $\mathcal{K}_2$ if and only if $\nabla^{can}k(X,X,X)=0$ for every tangent vector field $X$. As a consequence, $2div_{can}k+dtr_{can}k=0$ for all $k\in \mathcal{K}_2$, and every $k\in \mathcal{K}_{2,n}$ is divergence-free.
	
	 Differentiating the correspondence $k\in \{k\in\mathcal{K}_2\,|\,k>0\}\mapsto g_k\in \mathcal{K}$ described in \cite{ACM}, Proposition 9.8, along the path $(1-t)can+tk\in \mathcal{K}_{2,n}$, for arbitrary $k\in \mathcal{K}_{2,n}$, one checks that the tangent space of $\mathcal{K}_0$ at $can$ is the set of divergence-free traceless symmetric two-tensors of the form $h=k-can$, $k\in \mathcal{K}_{2,n}$. It is also straightforward to compute that, for every $k\in \mathcal{K}_2$,
	\begin{equation*}
		\Delta_{can} k + 2n\left(k - \frac{tr_{can}k}{n}can\right) = Hess_{can}(tr_{can}k),
	\end{equation*}
so that the tensors $h=k-can\in T_{can}\mathcal{K}_0$ with $k\in \mathcal{K}_{2,n}$ lie in an eigenspace of the (rough) Laplace operator $\Delta_{can}$ acting on symmetric two-tensors of $(\mathbb{RP}^n,can)$. 
	
	Thus, the infinitesimal directions of variations in $\mathcal{Z}$ corresponding to variations within $\mathcal{K}_0$ that start at $can$ constitute a finite dimensional space. Any non-trivial divergence-free, traceless symmetric two-tensor $h$ on $(\mathbb{RP}^n,g)$ that is $L^2$-orthogonal to $T_{can}\mathcal{K}_0$ (say, any other eigentensor of the Laplace operator in the space of divergence-free traceless symmetric two tensors, see  \cite{BoucettaSpectredesLaplacien}) generates metrics in $\mathcal{Z}\setminus \mathcal{K}$ via Theorem \ref{thmE}. Actually, even more is true: by Ebin's slice theorem, which constructs a slice for the action of the diffeomorphism group on the space of Riemannian metrics near $can$ whose tangent space at $can$ is precisely the space of divergence-free symmetric two-tensors with respect to $can$, the $L^2$-orthogonality of $h$ to $T_{can}\mathcal{K}_0$ implies that metrics $g_t$ generated by such $h$ are not isometric to any metric on $\mathcal{K}$ (see \cite{BerEbi}, Section 3).
\end{remark}

\subsection{Computation of the isometry group} \label{subseccomputation} According to Corollary \ref{corfirstpartthmB}, for any $f\in C^{\infty}_{odd}(\mathbb{S}^n)$, every element $A$ of the stabilizer group 
$$G_f:=\{R\in O(n+1):fR=f\},$$
maps $\iota_t(\mathbb{S}^n)$ into $\iota_t(\mathbb{S}^n)$. Since $A$ is an isometry of $\mathbb{R}^{n+1}$, the restriction of $A$ to $\iota_t(\mathbb{S}^n)$ is an isometry of $\iota_t(\mathbb{S}^n)$ with respect to the induced metric of $\mathbb{R}^{n+1}$. 

	Denote by $G_t$ the isometry group of $\iota_t(\mathbb{S}^n)$. While we have seen that $G_f\subseteq G_t$  (here, we do a minor abuse of notation and identify $A\in G_f$ with its restriction to $\iota_t(\mathbb{S}^n)$), our next aim is to prove that $G_f=G_t$ for every sufficiently small $t$, as soon as $f$ is $L^2$-orthogonal to the linear functions. This is what remains to prove Theorem \ref{thmB}.

	It is useful to keep in mind another identification. The radial perturbation $\iota_t: \mathbb{S}^n\rightarrow \iota_t(\mathbb{S}^n)$ induces an abstract metric $g_t$ on $\S^n$. Sometimes, it is convenient to consider the group $\hat{G}_t:=Iso(\mathbb{S}^n,g_t)$, which is conjugate to $\hat{G}_t$, with $G_t=\iota_t \hat{G}_t \iota_t^{-1}$. Since $A=\iota_t^{-1}A\iota_t$ (see Corollary \eqref{corfirstpartthmB}), also abusing notation, we will sometimes think of $G_f$ as a subgroup of $\hat{G}_t$.

	The first observation is that every intrinsic isometry of $\iota_t(\mathbb{S}^n)$ (and not only those in $G_f$) is the restriction of an extrinsic isometry of $\mathbb{R}^{n+1}$. To prove this, we will use rigidity results for ovaloids in  submanifold theory. To be more precise, if $n = 2$ we use the Cohn-Vossen Theorem [12], and if $n >3$ we use the Sacksteder Theorem [28] (for a modern proof see [13], Theorem 13.2). We observe that the first result requires positive Gaussian curvature, and the second applies if the Ricci curvature is positive.
	
\begin{lem}\label{lemintrinsicareextrinsic}
	Assume that $t$ is sufficiently small. For every $T\in G_t$ there exists $\tilde{T}\in Iso(\mathbb{R}^{n+1})$ with $\tilde{T}|_{\iota_t(\S^n)}=T$.
\end{lem}
\begin{proof}
    Let $\hat{j}=\hat{j}_t$ be an intrinsic isometry of $(\S^n,g_t)=\hat{G}_t$. Then, we have two isometric immersions of $(\S^n,g_t)$ into $\mathbb{R}^{n+1}$, namely, $\iota_t$ and $\iota_t\circ \hat{j}$.
    We claim that, as soon as $t$ is sufficiently small, there exists $\tilde{T}=\tilde{T}_t\in\text{Iso}(\R^{n+1})$ such that $\tilde{T}\circ\iota_t=\iota_t\circ \hat{j}$.
    Indeed, for $n=2$, the claim is a consequence of the Cohn-Vossen rigidity theorem for ovaloids, since $\iota_t(\mathbb{S}^2)$ is an embedded surface with positive Gaussian curvature for $t$ small enough. Similarly, in dimensions $n>2$, the claim follows from Sacksteder result applied to embedded hypersurfaces with positive Ricci curvature, which is the case for $t$ small enough. Since every element of $G_t$ is of the form $\iota_t \hat{j}\iota_t^{-1}$ for some $\hat{j}\in \hat{G}_t$, the conclusion follows.
\end{proof}

	Lemma \ref{lemintrinsicareextrinsic} implies that, as soon as $t$ is sufficiently small, we can do another identification and regard $G_t$ as the set of all isometries $T$ of $\mathbb{R}^{n+1}$ that map $\iota_t(\mathbb{S}^n)$ into itself. Thus, for every $T\in G_t$ there exist uniquely determined $A\in O(n+1)$ and $v\in \mathbb{R}^{n+1}$ such that $Tx=Ax+v$ for every $x\in \mathbb{R}^{n+1}$.
	
	Let $C^\infty_{odd,0}(\S^n)\subseteq C^\infty_{odd}(\S^n)$ be the subset of smooth odd maps of $\S^n$ that are $L^2$-orthogonal to all linear maps. We continue our analysis assuming that $f\in C^\infty_{odd,0}(\S^n)$. 

\begin{lem}\label{lemma f=fA}
    Let $f\in C^\infty_{odd,0}(\S^n)$. Given $t_k\rightarrow 0$, let $T_k\in G_{t_k}\subseteq \text{Iso}(\R^{n+1})$ be written as
    $$T_k(x)=A_k(x)+v_k$$
    for $A_k\in O(n+1)$ and $v_k\in \mathbb{R}^{n+1}$. If $\lim_k A_k=A\in O(n+1)$ exists, then $\lim_k v_k=0$ and $\lim T_k = A\in G_f$.
\end{lem}
\begin{proof}
	Write $\S_k:=\iota_{t_k}(\S^n)\subseteq\R^{n+1}$. We may assume that all $\S_k$ lie within a closed ball of radius $2$ centered at the origin in $\mathbb{R}^{n+1}$. Hence, the sequence $\{v_k\}$ is bounded in $\R^{n+1}$, and therefore every subsequence has a converging subsequence $\lim v_{t_m}=v\in \mathbb{R}^{n+1}$. We claim that $v=0$. In fact, by the hypothesis on $A_k$, the maps $T_{k_m}: x\mapsto A_{k_m}x+v_{k_m}$ converge to the map $T: x\mapsto Tx=Ax+v\in Iso(\mathbb{R}^{n+1})$. Since $T_{k_m}(\S_{k_m})=\S_{k_m}$, by passing to the limit we conclude that $T$ maps $\mathbb{S}^n$ onto $\mathbb{S}^n$, so that it is an element of $O(n+1)$. Therefore $v=0$, as claimed. This shows that every subsequence of $\{v_k\}$ has a subsequence converging to zero. It follows that $\lim v_k=0$.

	Therefore, $T_k\in G_{t_k}$ converges to $A\in O(n+1)$. It remains to prove that $A\in G_f$. 
	
	Let $H_k:\S^n\rightarrow\R$ be the mean curvature function of $\S_k$, that is, $H_k(x)$ is the mean curvature of $\mathbb{S}_k$ at the point $\iota_{t_k}(x)$. By Lemma \ref{applemvariationAH} of the Appendix, we have
    \begin{equation}\label{eqmiddleofclaim}
    	\lim_{k\to\infty}\Big(\frac{H_k(x)-n}{t_k}\Big) = -\Delta f(x)-nf(x), \quad\forall x\in\S^n. 
    \end{equation}
    Let $\hat{T}_k=\iota_{t_k}^{-1}T_k\iota_{t_k}\in \hat{G}_t$. By virtue of Lemma \ref{lemintrinsicareextrinsic}, the mean curvature $H_k(x)$ coincides with the mean curvature $H_k(\hat{T}_k(x))$, for every $x\in \mathbb{S}^n$. 
    
    \hspace{5mm}

\noindent \textbf{Claim}: For every $x\in \mathbb{S}^n$, $\lim\Big(\frac{H_k(\hat{T}_k(x))-H_k(A(x))}{t_k}\Big)=0$. 

\hspace{5mm}

	Assuming the claim, and using \eqref{eqmiddleofclaim} at the points $x$, $A(x)\in \mathbb{S}^n$, we obtain
    \begin{align*}
        -\Delta f(x)-nf&(x)=\lim_{k\to\infty}\Big(\frac{H_k(\hat{T}_k(x))-n}{t_k}\Big)\\
        &=\lim_{k\to\infty}\Big(\frac{H_k(\hat{T}_k(x))-H_k(A(x)) + H_k(A(x))-n}{t_k}\Big)\\
        &=
        -\Delta f(A(x))-nf(A(X))= -\Delta f^A(x)-nf^A(x),
    \end{align*}
    where $f^A=f\circ A$ (here we used that $A$ is an isometry of $(\mathbb{S}^n,can)$). Since $x$ is arbitrary, it follows that $(f-f^A)$ belongs to the kernel of $\Delta+n$ on $(\mathbb{S}^n,can)$, that is, the set of linear functions restricted to $\mathbb{S}^n$. This set and its $L^2$-orthogonal complement are invariant by the action of $O(n+1)$. Since, by hypothesis, $f$ is $L^2$-orthogonal to all linear functions, we conclude that $f^A$ has the same property. But then $f-f^A$ must vanish identically. In other words, $A\in G_f$, as we wanted to show.

    Let us now prove the Claim. Given $x\in \mathbb{S}^n$, denote by $d_k$ the distance in $(\mathbb{S}^n,can)$ between $A(x)$ and $\hat{T}_k(x)$. Clearly, $\lim d_k=0$.    Consider the normalized geodesic $\gamma(s)=\gamma_k(s)$ of $(\mathbb{S}^n,can)$ with $\gamma(0)=A(x)$ and $\gamma(d_k)=\hat{T}_k(x)$. 
    
    Let $H(t,x)=H_t(x)$ be the smooth function that computes the mean curvature of $\iota_t(\mathbb{S}^n)$ at the point $x$. Since the mean curvature of $\mathbb{S}^n=\iota_0(\S^n)$ is constant and equal to $n$, the Taylor expansion of the function $h(t,s)=H(t,\gamma(s))$ centered at $(0,0)$ up to second-order gives
    \begin{align*}
        H_{t_k}(\hat{T}_k(x))&=h(t_k,d_k) \\
        &=n-\big((\Delta f)(Ax)+nf(Ax)\big)t_k +\frac{1}{2}\partial_{t}\partial_t h(\theta t_k,\theta d_k)t^2_k\\
        &\quad+\partial_{t}\partial_s h(\theta t_k,\theta d_k)t_kd_k+\frac{1}{2}\partial_{s}\partial_s h(\theta t_k,\theta d_k)d^2_k,
    \end{align*}
    for some $\theta=\theta_k\in (0,1)$. Moreover, as $\partial_s\partial_s h(0,s)\equiv 0$, there exist constants $C$ and $\varepsilon>0$ such that
    \begin{equation}\label{ddH<Ct}
        |\partial_{s}\partial_s h(t,s)|\leq C|t|\quad\forall x\in\S^n,\quad \forall t\in[-\varepsilon,\varepsilon].
    \end{equation}
    Hence, as $k\rightarrow \infty$, $(t_k,d_k)\rightarrow (0,0)$ and 
    \begin{equation*}
    	 H_{t_k}(\hat{T}_k(x))=n-\big((\Delta f)(Ax)+nf(Ax)\big)t_k+o(t_k).
    \end{equation*}
    Similarly, the Taylor expansion at $0$ for the function $g(t)=H(t,Ax)$ gives
    $$H_{t_k}(Ax)=g(t_k)=n-\big((\Delta f)(Ax)+nf(Ax)\big)t_k+o(t_k)$$
    when $k\rightarrow \infty$. By comparing the two expansions, the claim is proven.
    \end{proof}

We choose, once and for all, an inner product in $Lie(Iso(\mathbb{R}^{n+1}))$, and denote by $\exp$ the exponential map of $Lie(Iso(\mathbb{R}^{n+1}))$ into $Iso(\mathbb{R}^{n+1})$. The next Lemma shows that the connected components of $G_f$ and $G_t$ are equal, as soon as $t$ is sufficiently small.

\begin{lem}\label{lemg=gt}
    If $f\in C^\infty_{odd,0}(\S^n)$ then $Lie(G_f)=Lie(G_t)$ for any sufficiently small $t\neq0$.
\end{lem}
\begin{proof}
    By Corollary \ref{corfirstpartthmB}, we have $G_f\subseteq G_t$, so that $Lie(G_f)\subseteq Lie(G_t)\subseteq Lie(Iso(\mathbb{R}^{n+1}))$ for all $t$. 
    
    Suppose, by contradiction, that there exists a sequence $t_k\rightarrow 0$ such that $Lie(G_f)\subsetneq Lie(G_{t_k})$  for every $k$. Then there exists $w_k\in Lie(G_k)\cap Lie(G_f)^\perp$ with $|w_k|=1$. After passing to a subsequence, we may assume that $w_k$ converges to a unit vector $w\in Lie(G_f)^\perp$. Fix $s\in \mathbb{R}$. The sequence $T_k=\exp(sw_k)\in G_k$ is a sequence of isometries of $\mathbb{R}^{n+1}$ mapping $\mathbb{S}_k$ into itself that converges to $A=\exp(sw)$, an isometry of $\mathbb{R}^{n+1}$ that maps $\mathbb{S}^n$ into itself. Thus $A\in O(n+1)$. By Lemma \ref{lemma f=fA}, possibly after passing to a further subsequence, we conclude that $\exp(sw)=A=\lim T_k\in G_f$. Since $s$ is arbitrary, $w\in Lie(G_f)$, a contradiction.
\end{proof}

Combining Lemmas \ref{lemintrinsicareextrinsic}, \ref{lemma f=fA} and \ref{lemg=gt}, we prove Theorem \ref{thmB}.
\begin{proof}[Proof of Theorem \ref{thmB}] Assume, by contradiction, that there exists a sequence $t_k\rightarrow 0$ with $G_k:=G_{t_k}\neq G_f$. By Lemma \ref{lemintrinsicareextrinsic}, since $G_f\subseteq G_k\subset Isom(\mathbb{R}^{n+1})$ for every $k$, in this case there exists $T_k\in G_k\setminus G_f$ of the form $T_k(x)=A_k(x)+v_k$, where $A_k\in O(n+1)$ and $v_k\in \mathbb{R}^{n+1}$. 
    
    After passing to a subsequence, we may assume that $\lim A_k=A\in O(n+1)$. By Lemma \ref{lemma f=fA}, we conclude that $\lim v_k=0$ and $\lim T_k=A\in G_f\subseteq G_k$. Thus, and since $T_k\in G_k\setminus G_f$, we have that $A^{-1}T_k\in G_k\setminus G_f$ forms a sequence that converges to the identity map. In other words, by changing $T_k$ for $A^{-1}T_k$, we can assume that our sequence $T_k$ is such that $\lim A_k=\text{Id}$ and $\lim v_k=0$.
    
    We claim that $A_k\notin G_f$ for all $k$. In fact, if not, then the composition $T_k A_k^{-1}$ (which is just the translation by $v_k$) is also an element of $G_k$. Since $G_k$ consists of maps that preserve a compact hypersurface of $\mathbb{R}^{n+1}$, this is possible only if $v_k=0$. Otherwise, the iterates would eventually move the hypersurface outside the ball that contains it. But in this case $T_k=A_k\in G_f$, a contradiction. This proves the claim.

    Consider the decomposition $Lie(SO(n+1))=Lie(G_f)\oplus Lie(G_f)^{\perp}$, with respect to the chosen inner product on $Lie(Iso(\mathbb{R}^{n+1}))$. Recall that $\exp$ denotes the exponential map of $Lie(Iso(\mathbb{R}^{n+1}))$ (which is a matrix group when seen inside the affine group). We have $\exp(Lie(SO(n+1)))=SO(n+1)$. The map
    \begin{align*}
        F:Lie(G_f)\times Lie(G_f)^{\perp}&\rightarrow SO(n+1)\\
        (u,w)&\mapsto \exp(u)\exp(w),
    \end{align*}
    has derivative equal to the identity at $(0,0)$, and therefore is a local diffeomorphism between a neighborhood of $(0,0)$ and a neighborhood of $F(0,0)=Id$. Since $\lim A_k=Id$, $A_k$ eventually lies in this small neighborhood.
    
    By Lemma \ref{lemg=gt}, $Lie(G_f)=Lie(G_k)$ for sufficiently large $k$. It follows that, for all sufficiently large $k$, we have $A_k=g_k\exp(s_kw_k)$, where $g_k\in G_f\cap SO(n+1) \subset G_k\cap SO(n+1)$ lies in a sufficiently small neighborhood of $Id$, $w_k\in Lie(G_f)^{\perp}$ is a unit vector, $s_k\neq 0$ is such that $s_kw_k$ lies in a sufficiently small neighborhood of zero in $Lie(G_f)^{\perp}$, and $\exp(s_k\omega_k)\in SO(n+1)\setminus G_f$ . Passing to a subsequence, we may assume that $g_k$ converges in the compact group $G_f$, the unit vectors $w_k$ converge to a unit $w\in Lie(G_f)^{\perp}$, and the small numbers $s_k$ converge to some number. In fact, since $A_k\rightarrow Id$, we have $s_k\rightarrow 0$. 
    
    Let $\overline{T}_k=g_k^{-1}T_k\in G_k$, $\overline{A}_k=\exp(s_kw_k)$ and $\overline{v}_k=g_k^{-1}v_k$. Since $\overline{A}_k\rightarrow Id$, by Lemma \ref{lemma f=fA}, $\overline{v}_k\rightarrow 0$ and $\lim\overline{T}_k=\lim\overline{A}_k=Id$. Notice again that $\overline{A}_k\notin G_f$ for any $k$ sufficiently large. 

    Given any $t\in\R$, let $n_k\in\Z$ be a sequence such that $\lim n_ks_k=t$. For the iterates $\overline{T}_k^{n_k}\in G_k$, we can find $\tilde{v}_{n_k}\in \R^{n+1}$ such that
    $$\overline{T}_k^{n_k}(x)=\overline{A}_k^{n_k}(x)+\tilde{v}_{n_k}=\exp(n_ks_kw_k)x+\tilde{v}_{n_k}.$$
    Since $\overline{T}_k^{n_k}\in G_k$ and
    $$\lim\overline{A}_k^{n_k}=\lim \exp(n_ks_kw_k)=\exp(tw),$$
    we conclude, again by Lemma \ref{lemma f=fA}, that $\tilde{v}_{n_k}\rightarrow 0$ and $\exp(tw)=\lim \overline{T}^{n_k}_k \in G_f$. Since $t$ is arbitrary, differentiation along the path $t\mapsto \exp(tw)$ at zero gives $w\in Lie(G_f)$. And this is the final contradiction.
\end{proof}

	To finish this section, we prove Theorem \ref{thmC}.
		
	\begin{proof}[Proof of Theorem \ref{thmC}]
	Let $f$ be any smooth non-zero odd function on $\mathbb{S}^n$, orthogonal to the linear functions, that depends only on the last coordinate $x_{n+1}$ (for instance, a zonal spherical harmonic of odd degree $>1$). Then $G_f$ contains a copy of $O(n)$, consisting of maps of the form $(x',x_{n+1})\mapsto (Rx',x_{n+1})$, $R\in O(n)$. In this case, Lemma \ref{O(n)subset G then O(n)=G} below shows that $G_f=O(n)$ since $f$ is non-constant.
	
	By Theorem \ref{thmB}, $f$ generates a smooth one-parameter family of embeddings of $\mathbb{S}^n$ in $\mathbb{R}^{n+1}$, each invariant by the above horizontal action of $O(n)$, and each containing a Zoll family of codimension one embedded minimal spheres $\{\Sigma_\sigma\}$, $\sigma\in \mathbb{RP}^n$, such that $A(\Sigma_{[v]})=\Sigma_{[Av]}$ for every $[v]\in \mathbb{RP}^n$ and every $A\in O(n)$.
	\end{proof}

    \begin{lem}\label{O(n)subset G then O(n)=G}
        Let $G\subseteq O(n+1)$ be a closed subgroup with $-Id\notin G$ and $O(n)\subseteq G$, where $O(n)\subseteq O(n+1)$ acts trivially in the last canonical vector $e_{n+1}$ and naturally in its orthogonal complement. Then $G=O(n)$.
    \end{lem}
    \begin{proof}
        For simplicity, we prove this for $n=2$, but the proof works in any dimension. In this case, the Lie algebra $Lie(O(2))$ consists of antisymmetric matrices with the last row and column being zero. 
        
        Observe that $G$ is a Lie subgroup of $O(3)$ as it is closed. In particular, $Lie(O(2))\subseteq Lie(G)\subseteq Lie(O(3))$. We also consider in $Lie(O(3))$ the inner product $tr(AB^t)$.
        
        If $Lie(G)\neq Lie(O(2))$ then we can take, possible after an orthogonal change of variables, $X\in Lie(G)\cap Lie(O(2))^\perp$ of the form 
        $$X=\begin{pmatrix}
0 & 0 & 1\\
0 & 0 & 0\\
-1& 0 & 0
\end{pmatrix}.$$
In this case, we easily verify that $[Lie(O(2)),X]+span\{X\}$ has dimension $3$, so $Lie(G)=Lie(O(3))$, which shows that $G=SO(3)$ since $-Id\notin G$. But in this case $O(2)$ is not contained in $G$, a contradiction.

On the other hand, if $Lie(G)=Lie(O(2))$ then the adjoint map $Ad_g(T)=g^{-1}Tg$ satisfies $Ad_g(Lie(O(2))\subseteq Lie(O(2))$ for any $g\in G$. This condition implies that $e_3$ is an eigenvector of $g\in G$. Observe that we cannot have $ge_3=-e_3$ since in this case there exists $T\in O(2)\subseteq G$ such that $-Id=Tg\in G$, which is a contradiction. Hence, $ge_3=e_3$ for any $g\in G$, and this shows that $G\subseteq O(2)$.
    \end{proof}

\section{Zoll metrics on $\S^2$ with discrete group of isometries}\label{sectionS2}
	Theorem \ref{thmB} reduces the problem of deciding which Lie subgroups of $O(n+1)$ can appear as the isometry group of the metrics constructed by Theorem \ref{thmA} to a purely algebraic problem. In this section, we solve this problem completely in dimension $n=2$ (see Theorem \ref{thmDbis} and Remark \ref{rmkobst} below).
	
	A consequence of Green's Theorem \cite{Gre} is that no isometry group of a non-round Zoll metric on $\mathbb{S}^2$ contains the antipodal map. The main result of this section is:
\begin{thm}\label{thmDbis}
    Let $G$ be a finite subgroup of $O(3)$ with $-Id\notin G$. Then there exists an odd smooth function $f:\S^2\rightarrow\R$ orthogonal to the linear maps and such that $G_f=G$.
\end{thm}

	As the proof of Theorem \ref{thmDbis} will be based on the classification of finite subgroups of $O(3)$, we start by recalling it.
	
	Two subgroups $G_1,G_2\subseteq O(3)$ are said to be {\it conjugate} if there exists $T\in O(3)$ such that $G_2=T^{-1}G_1T$.  This notion defines an equivalence relation among subgroups of $O(3)$. Since, for every $T\in O(3)$ and every $f\in C^{\infty}(\mathbb{S}^2)$, we have $G_{fT} = T^{-1}G_fT$, it is enough to prove Theorem \ref{thmD} for a representative of every conjugacy class of finite subgroups of $O(3)$.

	The conjugacy classes of finite subgroups of $O(3)$ are well understood. We present below the list of representatives of all these classes (for a proof, see, for example, the Appendix of \cite{Weyl}).

Any conjugacy class is represented by one and only one element of one of three families: Type $I$, $II$, and $III$.
The first family corresponds to the finite subgroups of $SO(3)$.
The list of representatives of each conjugacy class of Type $I$ is the following:
\begin{enumerate}
    \item The trivial group $\{Id\}$;
    \item The cyclic group $\Z_n$, $n>1$, generated by the map $(z,t)\in\C\times\R\mapsto(\zeta_{n}z,t)$ where $\zeta_n=e^{\frac{2\pi i}{n}}$;
    \item The dihedral group $D_n$, generated by the cyclic group $\Z_n$ described in the previous item and by the map $(z,t)\mapsto(\overline{z},-t)$;
    \item The group of orientation-preserving isometries of a regular tetrahedron centered at $0$, known as the tetrahedral group $\mathcal{T}$;
    \item The group of orientation-preserving isometries of a regular octahedron centered at $0$, known as the octahedral group $\mathcal{O}$;
    \item The group of orientation-preserving isometries of a regular icosahedron centered at $0$, known as the icosahedral group $\mathcal{I}$.
\end{enumerate}
Clearly, $(2)$ and $(3)$ comprise representatives of infinitely many different classes each. 

The orders of the groups above are
$$|\{Id\}|=1,\,|\Z_n|=n, \,|D_n|=2n, \,|\mathcal{T}|=12, \,|\mathcal{O}|=24,\text{ and } \,|\mathcal{I}|=60.$$

A subgroup $\tilde{G}\subseteq O(3)$ is of {\it Type $II$} when it is generated by a group $G$ of Type $I$ and by the map $-Id$. 
This family is not important to us.

Finally, the groups of {\it Type $III$} are obtained from the groups of Type $I$ by the following algebraic procedure. Given two finite subgroups $G_1\subseteq G_2\subseteq SO(3)$ with $|G_2|=2|G_1|$, then 
$$G_1[G_2:=\{T\in O(3)\,|\, \textit{$T\in G_1$ or $-T\in G_2\setminus G_1$}\}$$
is a group of order $|G_1[G_2|=|G_2|$. Using this notation, the list of representatives of the conjugacy classes of Type $III$ is
$$\{Id\}[\Z_2,\,\,\mathcal{T}[\mathcal{O},\,\,\Z_n[\Z_{2n},\,\,\Z_n[D_n,\text{ and } D_n[D_{2n}.$$
The last there comprise representatives of infinitely many different classes each.
 
	In order to prove Theorem \ref{thmDbis}, we will exhibit, for each of the above subgroups $G\subset O(3)$ of Type I and III, a polynomial $P=P_G$ in $\mathbb{R}^3$ that is a sum of homogeneous harmonic polynomials of odd degree $>1$, such that 
$$G_P:=\{T\in O(3)\,|\, PT=P\}=G.$$
	Then, the restriction $p:=P|_{\S^2}$ is a sum of eigenfunctions of $\Delta_{can}$ of odd degree bigger than one. In particular, $p$ is orthogonal to the first eigenspace of $\Delta_{can}$, that is, the set of (restrictions to $\mathbb{S}^2$ of) linear maps.
	
	Since $G_p=G_P=G$, a direct application of Theorem \ref{thmB} proves the existence of Zoll spheres in $\mathbb{R}^3$ whose groups of isometries are exactly $G$, as we wanted. 	

Before we perform the case-by-case analysis, let us state a useful elementary lemma, which will be used repeatedly without further comment.

\begin{lem}\label{lemsumpolynomial}
    If $P=\sum_{i=1}^k P_k$ is the sum of homogeneous polynomials of different degrees, then 
	\begin{equation*}	
		G_P=G_{P_1}\cap\ldots\cap G_{P_n}.
	\end{equation*}
\end{lem}
\begin{proof}
	The proof is a straightforward induction starting from the case $k=2$, which we prove as follows. Since it is clear that $G_{P_1}\cap G_{P_2}\subseteq G_P$, let $T\in G_P$. Then
	\begin{equation*}	
		P_1(x)-P_1(T(x))=P_2((T(x))-P_2(x),
	\end{equation*}
    which is an equality of homogeneous polynomials of different degrees. Hence
    \begin{equation*}
    	P_1(x)-P_1(T(x))=0=P_2((T(x))-P_2(x),
    \end{equation*}
    which implies that $T\in G_{P_1}\cap G_{P_2}$, as we wanted to show.
\end{proof}

\subsection{Symmetries of Type $III$}
We identify $\R^3\simeq \C\times\R$ and use coordinates $(z,t)\in\C\times\R$. Recall that $\zeta_n:=e^{\frac{2\pi i}{n}}$ denotes the $n^{th}$-root of unity.

\subsubsection{$D_n[D_{2n}$-symmetry}

For an integer $n>1$, let $F=F_n:\C\times\R\rightarrow\R$ be the harmonic homogeneous polynomial of odd degree
\begin{equation}\label{F of Dn[D2n}
   F(z,t)=F_n(z,t):=\left\{ \begin{array}{lr}
         z^n+\overline{z}^n&   \text{if $n$ is odd},\\
         (iz^n-i\overline{z}^n)t&  \,\text{if $n$ is even}.
    \end{array}\right. 
\end{equation}
We claim that $G_F=D_{n}[D_{2n}$. In fact, first observe that we have
$D_n\subseteq G_F$ and
$$F(-\zeta_{2n}^{2k-1}z,-t)=F(z,t),\quad\forall k\in\Z,\,\forall (z,t)\in\C\times\R, $$
so that $D_{n}[D_{2n}\subseteq G_F$.
Moreover, by the classification, the only finite subgroups of $O(3)$ that contain $D_n[D_{2n}$ are $D_{(2k-1)n}[D_{2(2k-1)n}$ for $k\in\N$.
Since $G_F$ is discrete (otherwise it would contain a copy of $SO(2)$, which it does not), we have $G_F=D_{(2k-1)n}[D_{2(2k-1)n}$ for some $k$. But then
$$F(z,t)=F(\zeta_{(2k-1)n}z,t),\quad\forall (z,t)\in\C\times\R,$$
which is possible if and only if $\zeta_{(2k-1)n}^n=1$. Hence, $k=1$ and $G_F=D_n[D_{2n}$, as claimed.

\subsubsection{$\Z_n[D_n$-symmetry}

For $n>1$ an integer, let $H=H_n:\C\times\R\rightarrow\R$ be the harmonic homogeneous polynomial of odd degree
\begin{equation}\label{H of Zn[Z2n}
   H(z,t)=H_n(z,t):=\left\{ \begin{array}{lr}
         iz^n-i\overline{z}^n&   \text{if $n$ is odd},\\
         -(z^n+\overline{z}^n)t&  \,\text{if $n$ is even}.
    \end{array}\right. 
\end{equation}
Observe that $H_n=F_n\circ T_n$, where $F_n$ is given by (\ref{F of Dn[D2n}) and $T_n(z,t)=(\xi z,t)$ for some $\xi\in \mathbb{C}$ solving $\xi^n=i$.
In particular,
$$G_{H_n}=T_n^{-1}G_{F_n}T_n.$$
This implies that $G_{H_n}$ has order $4n$ and is generated by $\Z_n[\Z_{2n}$ and the map $(z,t)\mapsto(\xi^{-2}\overline{z},-t)$.

	Notice also that the symmetry $(z,t)\mapsto(-\overline{z},t)$ belongs to $G_{H_n}$. Thus, in particular, $\Z_n[D_n\subseteq G_{H_n}$.

Consider then
$$P=H_n+H_{2n},$$
and observe, by direct inspection, that $\Z_n[D_n\subseteq G_P$. By Lemma \ref{lemsumpolynomial}, in particular we have $G_P\subseteq  G_{H_n}$. But since
$$P(-\zeta_{2n}z,-t)=H_n(z,t)-H_{2n}(z,t)\neq P(z,t),$$
we have $G_P=\Z_n[D_n$.

\subsubsection{$\Z_n[\Z_{2n}$-symmetry}

Recall the polynomials defined in (\ref{F of Dn[D2n}) and (\ref{H of Zn[Z2n}).
Consider
$$P:=H_n+F_{3n},$$
and observe that $\Z_n[\Z_{2n}\subseteq G_P\subseteq G_{H_n}$. But since $G_{H_m}$ is the group generated by $\Z_n[\Z_{2n}$ and the map $(z,t)\mapsto (\xi^{-2}\overline{z},-t)$ (as computed in 5.1.2), and
$$P(\xi^{-2}\overline{z},-t)=H_n(z,t)-F_{3n}(z,t)\neq P(z,t),$$
we conclude that $G_P=\Z_n[\Z_{2n}$.

\subsubsection{$\{Id\}[\Z_2$-symmetry} 
This is simply the group generated by the reflection $R_3(z,t)=(z,-t)$. Given the polynomials
$$P_1:=H_4+F_{12}\quad\text{ and }\quad P_2:=H_2+F_6,$$
we have already computed in 5.1.3 that $G_{P_1}=\Z_4[D_4$ and $G_{P_2}=\Z_2[D_2$. In particular, these groups act on the $\C$-plane and fix the $t$-coordinate.
Write $(z,t)=(x,y,t)\in\R^3$, and consider the orthogonal map $T(x,y,t)=(x,-t,y)$. A computation shows that 
\begin{equation*}
	\Z_4[D_4\cap T^{-1}(\Z_2[D_2)T=T^{-1}(\{Id\}[\Z_2)T
\end{equation*}
is the group generated by the reflection $R_2(x,y,t)=(x,-y,t)$.
Thus, considering the odd degree homogeneous harmonic polynomial $$P:=P_1T^{-1}+P_2,$$
we have $G_P=TG_{PT}T^{-1}=T(G_{P_1}\cap (T^{-1}G_{P_2}T))T^{-1}=\{Id\}[\Z_2.$

\subsubsection{$\mathcal{T}[\mathcal{O}$-symmetry}

The group $\mathcal{T}[\mathcal{O}$ consists of the isometries of a regular tetrahedron, including the ones that reverse orientation (the group $\mathcal{T}$ contains only those that preserve orientation).
It is well known that $\mathcal{T}[\mathcal{O}\cong S_4$, where $S_4$ is the group of permutations of the four elements.

	More generally, we consider the regular $(n+1)$-simplex, whose symmetry group is isomorphic to the group of permutations $S_{n+2}$, and describe an odd degree spherical harmonic $f:\S^n\rightarrow\R$, orthogonal to the linear maps, such that $G_f$ is precisely the symmetry group of the regular $(n+1)$-simplex. \\

The group $S_{n+2}$ is the set of bijections of the set $\{1,\ldots,n+2\}$ onto itself.
The elements of $S_{n+2}$ will be denoted with Greek letters.

Let $\{e_1,e_1,\ldots,e_{n+2}\}$ be the canonical orthonormal basis of $\R^{n+2}$ and consider the representation $S_{n+2}\rightarrow O(n+2)$ given by 
$$\sigma(e_i)=e_{\sigma(i)},\quad\forall i.$$
This representation is faithful, but not irreducible, because  $S_{n+2}$ acts trivially in the direction $N:=(1,1,\ldots,1)$. 
For this reason, let 
$$\V^{n+1}:=N^\perp\subseteq\R^{n+2},$$ 
and denote by $O(\V^{n+1})\subseteq O(n+2)$ the group of orthogonal matrices that fix $N$. Then, we have an induced faithful representation $S_{n+2}\rightarrow O(\V^{n+1})$, and from now on we think of $S_{n+2}$ as a subgroup of $O(\V^{n+1})$.
Geometrically, the maps in $S_{n+2}\subseteq O(\V^{n+1})$ are characterized as the orthogonal maps of $\V^{n+1}$ that preserve the regular $(n+1)$-simplex whose vertices are 
\begin{equation}\label{eqverticessimplex}
	v_i:=\frac{1}{\sqrt{(n+1)(n+2)}}\left((n+2)e_i-N\right)\in \mathbb{V}^{n+1}.
\end{equation}
In the rest of this argument, we use $\S^n$ to denote the unit sphere centered at the origin of $\V^{n+1}$. Notice that $v_i\in \mathbb{S}^{n}$.

For an odd integer $m>1$, consider the polynomial 
$$F(x_1,\ldots,x_{n+2})=\sum_{k=1}^{n+2}x_k^m,$$
and let $f:=F|_{\S^n}$ be the restriction of $F$ to $\S^n\subseteq\V^{n+1}$.
Observe that $F|_{\V^{n+1}}$ is a homogeneous polynomial of degree $m$ that is harmonic on $\mathbb{V}^{n+1}$. In fact, using the formula for the Laplacian of a submanifold (in our case, $\V^{n+1}\subseteq\R^{n+2}$ is minimal), we obtain 
\begin{align*}
    \big(\Delta_{\V^{n+1}}(F|_{\V^{n+1}})\big)(x&)=\big(\Delta_{\R^{n+2}}F\big)(x)-Hess F(N,N)(x)\\
    &=m(m-1)\Big(\sum_{k=1}^{n+2}x_k^{m-2}\Big)-\frac{d^2}{dt^2}\Big|_{t=0}\Big(\sum_{k=1}^{n+2}(x_k+t)^{m}\Big)=0.
\end{align*}
Thus $f$ is an eigenfunction of the Laplacian of $\S^n$ that is moreover orthogonal to the linear maps, because its degree is $m>1$.

Clearly, $S_{n+2}\subseteq G_f$. In order to prove the opposite inclusion, we compute the critical points of $f$, which must be permuted by any $T\in G_f$. 

A point $x=(x_1,\ldots, x_{n+2})\in\S^n$ is a critical point of $f$ if and only if there exist $\mu,\nu\in\R$ such that
\begin{equation}\label{critical point symmetry S-(n+2)}
    Q(x_i):=mx_i^{m-1}-2\mu x_i-\nu=0,\quad\forall i.
\end{equation}
\begin{lem}\label{Q has at most two roots}
    For any $c,d\in\R$ constants, the polynomial 
    $$Q(y)=Q_{c,d}(y):=y^{m-1}-cy-d,$$
    has at most two real roots.
\end{lem}
\begin{proof}
    If all the real roots are zero, then there is noting to prove.
    Assume that $r\neq0$ is a real root.
    Observe that 
    $$Q(x)=Q(x)-Q(r)=\tilde{Q}\big(\frac{x}{r}\big)(x-r)r^{-m+2},$$
    where
    $$\tilde{Q}(y)=-r_1^{m-2}c+\sum_{j=0}^{m-2}y^j.$$
    Write $m=2n+1$. We have
    \begin{align*}
        \tilde{Q}'(y)&=\sum_{j=0}^{2n-2}(j+1)y^j=\Big(\sum_{k=1}^{n-1}(ky^{2k}+2ky^{2k-1}+ky^{2k-2})\Big)+ny^{2n-2}\\
        &=ny^{2n-2}+\sum_{k=1}^{n-1}ky^{2k-2}(y+1)^2>0,\quad \forall y\in\R.
    \end{align*}
    Hence, $\tilde{Q}$ has at most one real root.
\end{proof}
\begin{lem}
The set of critical points of $f$ decomposes as 
$$\bigcup_{1\leq k\leq\left \lfloor{\frac{n+2}{2}}\right \rfloor } (C_k^{+}\cup C_k^{-}),$$
where 
$C_k^-=-C_k^+$ and 
$C_k^+$ has only one element (modulo permutation of its coordinates).
Moreover, there are numbers $t_1>\ldots >t_{\left \lfloor{\frac{n+2}{2}}\right \rfloor }>0$ such that $f(C_k^+)=t_k=-f(C_k^-)$.
\end{lem}
\begin{proof}
    Let $x=(x_1,\ldots,x_{n+2})\in\S^n\subset \mathbb{V}^{n+1}$ be a critical point of $f$.
    Observe that it is not possible that $x_1=\ldots=x_{n+2}$, because $x$ is orthogonal to $N$. 
    
    As observed previously, there are $\mu,\nu\in\R$ satisfying (\ref{critical point symmetry S-(n+2)}).
    Lemma \ref{Q has at most two roots} implies that there exists $k$ such that $x$ is (up to permutation of indices) such that 
    $$x_1=\ldots=x_k=a_k\neq b_k=x_{k+1}=\ldots=x_{n+2}.$$
    Since $\sum x_i^2=1$, we determine
    $$a_k:=\pm\sqrt{\frac{n+2-k}{(n+2)k}} \quad\text{ and }\quad b_k:=\mp\sqrt{\frac{k}{(n+2)(n+2-k)}}.$$
    Let $C_k^+$ and $C_k^-$ be the set of points where $a_k>0>b_k$ and $a_k<0<b_k$, respectively. It is not difficult to verify that these points are indeed critical points, for we can indeed solve $\mu$ and $\nu$ in terms of $a_k$ and $b_k$ explicitly.
    
    Finally, observe that
    \begin{align*}
        t_k:&=f(C_k^+)=ka_k^m+(n+2-k)b_k^m\\
        &=k\Big(\frac{n+2-k}{(n+2)k}\Big)^{\frac{m}{2}}-(n+2-k)\Big(\frac{k}{(n+2)(n+2-k)}\Big)^{\frac{m}{2}},
    \end{align*}
    is decreasing in $k$, because the function
    $$h(y)=y\Big(\frac{n+2-y}{y}\Big)^{\frac{m}{2}}-(n+2-y)\Big(\frac{y}{n+2-y}\Big)^{\frac{m}{2}},$$
    is strictly decreasing on the interval $(0,n+2)$.
\end{proof}

We are now ready to show that $G_f=S_{n+2}$.
If $T\in G_f$, then $T$ preserves the values of $f$ and permutes its critical points. In particular, $T$ permutes the points of $C_1^+$. But this set is precisely the set of vertices of the regular $(n+1)$-simplex in $\mathbb{V}^{n+1}$, namely, the vectors $v_i$ as in \eqref{eqverticessimplex}. Given the geometric definition of $S_{n+2}$, $T$ belongs to $S_{n+2}$, as we wanted to show.

\subsection{Symmetries of Type $I$}
We analyze next the groups of Type $I$.

\subsubsection{$\mathcal{T}$-symmetry}

Under the identification $\mathcal{T}[\mathcal{O}\cong S_4$, the group $\mathcal{T}$ is identified with $A_4$, the alternating group of order $4$. Following the discussion of the case of $\mathcal{T}[\mathcal{O}$ symmetry, we view $A_{n+2}\subset S_{n+2}$ as the group of orientation-preserving isometries of the regular $(n+1)$-simplex in $\mathbb{V}^{n+1}$, and we find an odd degree spherical harmonic $f:\S^n\rightarrow\R$ (orthogonal to the linear maps) with $G_f=A_{n+2}\subseteq S_{n+2}$, for every $n\geq 2$.\\

Consider the Vandermonde polynomial $V:\R^{n+2}\rightarrow\R$,
$$V(x)=V(x_1,\ldots,x_{n+2}):=\prod_{i<j}(x_i-x_j).$$
$V$ is a homogeneous polynomial of degree $\binom{n+2}{2}$, which is odd only when $n\equiv 0,1$ $(\text{mod }4)$.
For this reason, consider the homogeneous polynomial
$$\tilde{V}(x)=\tilde{V}(x_1,\ldots,x_{n+2}):=\left\{ \begin{array}{lr}
         V(x)&   \text{if } n\equiv 0,1\text{ (mod } 4),\\
         (x_1^3+\ldots+x_{n+2}^3)V(x)&  \,\text{if } n\equiv 2,3\text{ (mod } 4).
    \end{array}\right.$$
\begin{lem}\label{lemauxvnadermond}
    $\tilde{V}|_{\V^{n+1}}$ is a harmonic polynomial of odd degree $>1$.
\end{lem}
\begin{proof}
    Every alternating polynomial in $n+2$ variables is divisible by $x_i-x_j$, for every $i<j$ . Thus, $V$ has the smallest degree among non-trivial alternating polynomials in $n+2$ variables. But $\Delta_{\R^{n+2}} V$ is an alternating polynomial in $n+2$ variables, with degree $deg(V)-2$ if it is not zero. Therefore $\Delta_{\R^{n+2}}V$ must vanish, that is, $V$ is harmonic on $\mathbb{R}^{n+2}$.

	Let us analyze the restriction of $\tilde{V}$ to $\mathbb{V}^{n+1}$ (we follow the notation of 5.1.5). Consider first the case, $n\equiv 0,1$ $(\text{mod }4)$.
    By the formula of the Laplacian of a submanifold,
    \begin{align*}
    \big(\Delta_{\V^{n+1}}(\tilde{V}|_{\V^{n+1}})\big)(x&)=\big(\Delta_{\R^{n+2}}V\big)(x)-Hess V(N,N)(x)\\
    &=0-\frac{d^2}{dt^2}\Big|_{t=0}\Big(\prod_{i<j}(x_i+t-x_j-t)\Big)=0.
\end{align*}
    In the other case, $n\equiv2,3$ $(\text{mod }4)$, we have
    \begin{align*}
        \big(\Delta_{\V^{n+1}}(\tilde{V}|_{\V^{n+1}})\big)(x&)=\big(\Delta_{\R^{n+2}}\tilde{V}\big)(x)-Hess \tilde{V}(N,N)(x)\\
    &=6\left(\sum_{i=1}^{n+2}x_i\right)V(x) + 6\left(\sum_{i=1}^{n+2}x_i^2\frac{\partial}{\partial x_i}V(x)\right)\\
    &\quad-\frac{d^2}{dt^2}\Big|_{t=0}\left(\left(\prod_{i<j}(x_i+t-x_j-t)\right)\left(\sum_{i=1}^{n+2}(x_i+t)^3\right)\right)\\
    &=6\left(\sum_{i=1}^{n+2}x_i^2\sum_{j\neq i}\frac{V(x)}{x_i-x_j}\right)-6V(x)\Big(\sum_{i=1}^{n+2}x_i\Big)\\
    &=6V(x)\sum_{i<j}\frac{x_i^2-x_j^2}{x_i-x_j}=6V(x)\sum_{i<j}(x_i+x_j)=0,
    \end{align*}
where we used repeatedly that $\sum_{i=1}^{n+2} x_i=0$ for points $x=(x_1,\ldots,x_{n+2})\in \mathbb{V}^{n+1}$, by definition of this space.
\end{proof}

It follows from Lemma \ref{lemauxvnadermond} that $\tilde{V}|_{\S^n}$ is an eigenfunction of $\Delta_{\S^n}$ of degree $>1$, and therefore is orthogonal to the linear maps.

For an odd integer $m>1$ with $m\neq\deg(\tilde{V})$, consider the odd degree polynomial $\map{P}{\V}{n+1}{\R}$ given by
$$P(x):=\tilde{V}(x)+F(x),$$
where $F=\sum_ix_i^m$.
Then $A_{n+2}\subseteq G_P\subseteq G_F=S_{n+2}$, as we verified in 5.1.5.
However, if $T\in G_P\subseteq S_{n+2}$, then
$$V(T(x))=V(x),\quad\forall x\in\R^{n+2},$$
and this implies that $T\in A_{n+2}$. We conclude that $G_P=A_{n+2}$.

\subsubsection{$D_n$-symmetry}
Recall the polynomial $F_n$ defined in (\ref{F of Dn[D2n}), and consider 
$$P:=F_n+F_{2n},$$
which satisfies $D_n\subseteq G_P\subseteq G_{F_n}=D_n[D_{2n}$.
However, 
$$P(-\zeta_{2n}z,-t)=F_n(z,t)-F_{2n}(z,t)\neq P(z,t).$$
Hence, $D_n=G_P$.

	For use in the next case, we observe that, in fact, the same argument shows that $D_n=G_{F_n}\cap G_{F_{2kn}}$ for any $k\geq1$.

\subsubsection{$\Z_n$-symmetry}

Recall the polynomials defined in (\ref{F of Dn[D2n}) and (\ref{H of Zn[Z2n}), and consider
$$P:=F_n+H_{2n}+F_{4n}.$$
We have $\Z_n\subseteq G_P\subseteq G_{F_n}\cap G_{F_{4n}}$, and in 5.2.2 we observed that $G_{F_n}\cap G_{F_{4n}}=D_n$.
But since
$$P(\overline{z},-t)=F_n(z,t)-H_{2n}(z,t)+F_{4n}(z,t)\neq P(z,t),$$
we conclude that $G_P=\Z_n$.

\subsubsection{$\mathcal{O}$-symmetry}

The octahedral group $\mathcal{O}\subseteq SO(3)$ is a group of order $24$ consisting of the orientation-preserving symmetries of the octahedron, or, equivalently, of the orientation-preserving symmetries of its dual solid, namely, the cube. Observe that the full group of symmetries of the octahedron contains the antipodal map, that is, it is a group of type $II$ (in contrast to the tetrahedron).

Consider the polynomial given by
$$P(x,y,z)=xyz(x^2-y^2)(y^2-z^2)(z^2-x^2),$$
which is harmonic and homogeneous of degree $9$.

In \cite{Meyer}, it was observed that $P$ is preserved by $\mathcal{O}$, for the following intuitive reason: Consider a cube with center at the origin in such way that the coordinates axes are orthogonal to its faces. Any $T\in \mathcal{O}$ sends faces to faces, and pairs of opposite edges into pairs of opposite edges. Thus, any such $T$ permutes the two dimensional linear subspaces parallel to the faces and the ones containing opposite edges. Now, $P$ is simply the product of the degree one polynomials that define each of these nine planes. Hence, every such $T$ leaves $P$ invariant, that is, $\mathcal{O}\subseteq G_P$.
	
	Going through the list of discrete subgroups of $O(3)$, we see that any group containing $\mathcal{O}$ is either $\mathcal{O}$ or contains $-Id$. Since $P$ has odd degree, the only possibility is $G_P=\mathcal{O}$.

\subsubsection{$\mathcal{I}$-symmetry}

Let $\varphi=\frac{1+\sqrt{5}}{2}$ be the golden ratio and consider the degree 13 homogeneous polynomial
\begin{align*}
    P(x,y,z)&=xyz(\varphi x+\varphi^{-1} y+z)(-\varphi x+\varphi^{-1} y+z)(\varphi x-\varphi^{-1} y+z)\\
    &\quad (\varphi x+\varphi^{-1} y-z)(x+\varphi y+\varphi^{-1}z)(-x+\varphi y+\varphi^{-1}z)\\
    &\quad(x-\varphi y+\varphi^{-1}z)(x+\varphi y-\varphi^{-1}z)(\varphi^{-1}x+y+\varphi z)\\
    &\quad (-\varphi^{-1}x+y+\varphi z)(\varphi^{-1}x-y+\varphi z)(\varphi^{-1}x+y-\varphi z).
\end{align*}
As observed in \cite{Meyer}, $P$ is a harmonic polynomial preserved by $\mathcal{I}$.
The intuition is similar to the case of the octahedral group:  each factor of $P$ vanishes precisely on the orthogonal complement of a set of axes of symmetry of the icosahedron that are permuted by elements of $\mathcal{I}$. Thus $\mathcal{I}\subseteq G_P$. Since $-Id\notin G_P$,  we conclude, again by the classification, that $G_P=\mathcal{I}$.

\subsubsection{$\{Id\}-$symmetry}
This is the generic case. See Corollary 4.71 of \cite{Bes} for a proof.

\begin{remark}\label{rmkobst}
    We can now fully describe the group of isometries $G$ of a Zoll metric $g$ on $\S^2$ near the canonical metric $can$.
    
    Firstly, Ebin's Theorem 8.1 in \cite{EbinThemanifoldofRiemannianmetricsProcSymp} shows that, up to a conjugation of a diffeomorphism of $\S^2$, $G$ is a closed subgroup of $O(3)$.
    Moreover, Green's Theorem \cite{Gre} shows that if $(\S^2,g)$ is not round, then $-Id\notin G$.

	The isometry group $G$ is then either finite, one-dimensional or $O(3)$. In fact, if $dim(G)\geq 2$, then $g$ is homogeneous and therefore it has constant positive curvature, so that $G= Iso(\mathbb{S}^2,can)= O(3)$. 
	
	Suppose that $dim(G)=1$. Then $G$ contains a copy of $SO(2)\simeq \mathbb{S}^1$. According to \cite{Bes}, Corollary 4.16, there exists coordinates $(r,\theta)$, with $r\in (0,\pi)$, $\theta\in (-\pi,\pi)$, such that
	\begin{equation*}
		g = (1+h(\cos(r)))^2dr^2 + \sin(r)^2d\theta^2,
	\end{equation*}   
where $h:[-1,1]\rightarrow (-1,1)$ is a non-zero smooth odd function such that $h(1)=0$ (note that this chart covers $\mathbb{S}^2$ except for a geodesic arc joining the fixed points of the isometric circle action). From this explicit representation of $g$, it is easy to check that $G$ contains a copy of $O(2)$, generated by the maps $(r,\theta)\mapsto (r,\theta+\alpha)$, $\alpha \in \mathbb{R}$ and the orientation-reversing map $(r,\theta)\mapsto (r,-\theta)$. 
In this case, Lemma \ref{O(n)subset G then O(n)=G} implies $G=O(2)$.
	
	In conclusion, the problem of classifying which Lie groups can appear as the isometry group of some Zoll metric on $\mathbb{S}^2$ near $can$ (up to conjugation of diffeomorphisms of $\mathbb{S}^2$) is solved: the possibilities are precisely $O(3)$ (the constant curvature metrics), the standard $O(2)$ (for instance, the metrics of Theorem \ref{thmC}) and each discrete subgroup of $O(3)$ that does not contain $-Id$ (for instance, the metrics of Theorem \ref{thmD}). Notice, moreover, that the Zoll metric $g$ with prescribed isometry group $G$ can be chosen arbitrarily close to $can$.
\end{remark}

\section*{Appendix}
\renewcommand{\thesubsection}{\Alph{subsection}}
\setcounter{subsection}{0}

\subsection{Second fundamental form star-shaped embeddings}

We collect here the formulae that describe the shape operator and the mean curvature of a star-shaped Euclidean sphere, parametrized as in \eqref{eqisom}, and their first variations under perturbations. All differential operations on functions $\psi$ on $\mathbb{S}^n$ are performed with respect to the canonical metric of $\S^n$.

Given $\psi\in C^{\infty}(\S^n)$, let $\iota=\map{\iota_\psi}{\S}{n}{\R^{n+1}}$ be the embedding defined by $\iota(x)=e^{\psi(x)}x$. The metric induced by $\iota$ on $\mathbb{S}^n$ is denoted $g_\psi$. Let $N=\map{N_\psi}{\S}{n}{\S^n}$ assign to each $x\in \mathbb{S}^n$ the unit outward-pointing normal vector of the sphere $\iota(\S^n)\subseteq\R^{n+1}$ at the point $\iota(x)$. 
\begin{lem}\label{applemN}
\begin{equation*} 	N(x)=\frac{x-\nabla\psi(x)}{(1+|\nabla\psi(x)|^2)^{\frac{1}{2}}}.
\end{equation*}
\end{lem}
\begin{proof}
	A straightforward computation, given that $x$ is a unit vector orthogonal to $T_x\mathbb{S}^n$.
\end{proof}
	Let $A=A_\psi:T\S^n\rightarrow T\S^n$ be the shape operator (i.e. the Weingarten map) with respect to $N$. We use the convention $d\iota\circ A=dN$.
Its pointwise trace is the mean curvature function $H=\map{H_\psi}{\S}{n}{\R}$ of $\iota(\S^n)$, that is, $H(x) = tr_{g_{\psi}} A_x$. 

In the next Lemma, $Hess\,\psi$ is regarded, at each point $x\in \mathbb{S}^n$, as the symmetric endomorphism of $T_x\mathbb{S}^n$ corresponding to $(\nabla d\psi)_x$, and $I$ denotes the identity endomorphism.

\begin{lem}\label{applemAH} The shape operator and the mean curvature are given by
\begin{equation*}
A=\frac{e^{-\psi}}{(1+|\nabla\psi|^2)^{\frac{1}{2}}}\Big(-Hess\,\psi+I+\frac{d\psi(Hess\,\psi(\cdot))}{1+|\nabla\psi|^2}\nabla\psi\Big),
\end{equation*}
\begin{equation*}
H=\frac{e^{-\psi}}{(1+|\nabla\psi|^2)^{\frac{1}{2}}}\bigg(-\Delta\psi+n+\frac{Hess\,\psi(\nabla\psi,\nabla\psi)}{1+|\nabla\psi|^2}\bigg).
\end{equation*}
\end{lem}
\begin{proof}
    We compute at a given point $p\in\S^n$. For every $v\in T_p\S^n$,
    \begin{equation*}
    	d\iota(Av)=dN(v)=\frac{v-\nabla_{v}\nabla\psi+\inner{\nabla\psi}{v}x}{(1+|\nabla\psi|^2)^{\frac{1}{2}}}-\frac{\inner{\nabla_{v}\nabla\psi}{\nabla\psi}}{(1+|\nabla\psi|^2)^{\frac{3}{2}}}(x-\nabla\psi).
    \end{equation*}
    After rearranging terms and introducing some notation, we obtain
    \begin{equation*}
    	d\iota(Av)=\inner{V}{\nabla\psi}x+V=d\iota(e^{-\psi}V),
    \end{equation*}
    where 
    \begin{equation*}
    	V:=\frac{1}{(1+|\nabla\psi|^2)^{\frac{1}{2}}}\Big(-\nabla_v\nabla\psi+v+\frac{\inner{\nabla_v\nabla\psi}{\nabla\psi}}{1+|\nabla\psi|^2}\nabla\psi\Big).
    \end{equation*}
    Hence
    \begin{equation*}
    	Av=e^{-\psi}V=\frac{e^{-\psi}}{(1+|\nabla\psi|^2)^{\frac{1}{2}}}\Big(-Hess\,\psi(v)+v+\frac{d\psi(Hess\,\psi(v))}{1+|\nabla\psi|^2}\nabla\psi\Big),
    \end{equation*}
    as claimed.
    
    In order to compute the mean curvature, let $e_1,\ldots, e_n\in T_p\S^n$ be an orthonormal basis of eigenvectors of $Hess\,\psi$ (with respect to the canonical metric), so that $\nabla_{e_i}\nabla\psi=\lambda_ie_i$ for certain eigenvalues $\lambda_i\in\R$.
    
    Writing $a_i=\inner{\nabla\psi}{e_i}$, we have
\begin{align*}
	g_\psi(Ae_i,e_j) & =\inner{d\iota(Ae_i)}{d\iota(e_j)} \\
	& = -\Big\langle\frac{(\lambda_i-1)e_i-a_ix}{(1+|\nabla\psi|^2)^{\frac{1}{2}}}-\frac{\lambda_ia_i}{(1+|\nabla\psi|^2)^{\frac{3}{2}}}(x-\nabla\psi),e^\psi(a_jx+e_j)\Big\rangle,
\end{align*}
	so that
    \begin{equation}\label{eqgAeiej}
        g_\psi(Ae_i,e_j)=-\frac{e^\psi}{(1+|\nabla\psi|^2)^{\frac{1}{2}}}\big((\lambda_i-1)\delta_{ij}-a_ia_j\big)
    \end{equation}
    for all indices $i$, $j$. On the other hand, 
    \begin{equation*}
    	g_{ij}=g_\psi(e_i,e_j)=e^{2\psi}(\delta_{ij}+a_ia_j).
    \end{equation*}
    Denote by $g^{ij}$ the coefficient of the inverse matrix of $G=(g_{ij})$. A straightforward computation shows that 
    \begin{equation}\label{eqgijemcima}
    	g^{ij}=e^{-2\psi}\Big(\delta_{ij}-\frac{a_ia_j}{1+|\nabla\psi|^2}\Big).
    \end{equation}
    Since the mean curvature can be computed by    \begin{equation*}
    	H=\sum_{i,j=1}^{n}g^{ij}g_\psi(Ae_i,e_j),
    \end{equation*}
    the final formula for $H$ is obtained by combining \eqref{eqgAeiej} and \eqref{eqgijemcima}, and observing that $\nabla\psi=\sum a_ie_i$, $\Delta \psi=\sum\lambda_i$, and $Hess\, \psi (\nabla \psi,\nabla\psi)= \sum \lambda_ia_i^2$. 
\end{proof}

Let $\iota_t:\S^n\rightarrow\R^{n+1}$, $t\in(-\varepsilon,\varepsilon)$, be a one-parameter family of star-shaped embeddings given by $\iota_t(x)=e^{\psi(t,x)}x$, where $\psi:(-\varepsilon,\varepsilon)\times\S^n\rightarrow\R$ is smooth and $\psi(0,x)=x$ for all $x\in \mathbb{S}^n$. 

\begin{lem}\label{applemvariationAH}
  For $f(x)=\partial_t\psi(0,x)$,  we have
    \begin{align*}
        \frac{\partial}{\partial t}\Big|_{t=0} A_t&=-Hess\,f-fI,\\
        \frac{\partial}{\partial t}\Big|_{t=0} H_t&=-\Delta f-nf.
    \end{align*}
\end{lem}
\begin{proof}
	A straightforward computation using the formulae in Lemma \ref{applemAH}.
\end{proof}

\begin{remark}\label{rmkHeR}
	As one can observe from the Gauss equation, Corollary \ref{applemvariationAH} confirms that the expansion of $H_t$ in $t$ has the same first-order term (up to a constant factor) as the scalar curvature; see Section 9.3 of \cite{ACM}. This expansion was a key ingredient in the proof of Theorem C therein. Thus, if we combine it with the fact that the mean curvature is preserved by intrinsic isometries of $(\S^n,g_{\psi_t})$ (which is a direct consequence of Lemma \ref{lemintrinsicareextrinsic}), we can follow the steps in the proof of Theorem C of \cite{ACM} to show that, by the same judicious choice of initial smooth odd function on $\mathbb{S}^n$, Theorem \ref{thmA} yields isometric embeddings of metrics in $\mathcal{Z}$ with trivial isometry group. 

	In any case, Theorem \ref{thmB} contains more information. Indeed, it guarantees that, for \textit{any} given $f\in C^{\infty}_{odd}(\mathbb{S}^n)$ that is $L^2$-orthogonal to linear functions and such that $G_f=\{Id\}$, there exists a one-parameter family of metrics $g_t\in \mathcal{Z}$, each one induced by a star-shaped perturbation of $\mathbb{S}^n$, that have trivial isometry group for all sufficiently small $t\neq 0$.
\end{remark}

\subsection{Comparison between constructions} Given a smooth odd function $f$ on $\mathbb{S}^n$, we have now two ways of constructing smooth one-parameter families $g_{t}\in \mathcal{Z}$ out it, so that $g_t=(1 + tf)can + o(t)$ as $t$ goes to zero. Namely, we can use either Theorem A of \cite{ACM} or Theorem \ref{thmA}. We can show that, generically, these families are distinct. More precisely, we give a criterion to distinguish both families in terms of the Hessian of $f$ on $(\mathbb{S}^n,can)$.

\begin{prop}
    Let $f$ be a smooth odd function on $\S^n$, $n\geq4$. Let  $g_t=\iota_t^*can$ and $\overline{g}_t=e^{\rho_t}can$ be the associated one-parameter families of Riemannian metrics in $\mathcal{Z}$ given by Theorem \ref{thmA} and Theorem A of \cite{ACM}, respectively.
    
    Let $U$ be any open neighborhood of a given point $p\in \mathbb{S}^n$. If $\text{Hess}\,f(p)$ does not have an eigenvalue of multiplicity at least $n-1$, then, for all sufficiently small $t\neq 0$, 
    \begin{itemize}
        \item[$(i)$] $(U,g_t|_U)$ is not conformally flat; and
        \item[$(ii)$] $(U,\overline{g}_t|_U)$ does not have an isometric immersion in $\R^{n+1}$.
    \end{itemize}
\end{prop}
\begin{proof}
    When $n\geq 4$, it is well known that a conformally flat hypersurface of $\mathbb{R}^{n+1}$ has a principal curvature of multiplicity at least $n-1$ (see \cite{Cartan} or \cite{NM}). 
    
    By the assumption on $f$, Lemma \ref{applemvariationAH} implies that the principal curvatures of $A_t(p)$ do not have this property, for any sufficiently small $t$. This proves item $(i)$.
    
    For the second item, let us consider the Ricci endomorphism of the metric $\overline{g}_t=e^{2\rho_t}can$. By a standard computation, 
    \begin{equation}\label{eqricconform}
    	Ric_t=e^{-2\rho_t}\big(Ric_0-(n-2)[Hess\, \rho_t-d\rho_t(\cdot)\nabla\rho_t]-(\Delta\rho_t+(n-2)|d\rho_t|^2)I\big),
    \end{equation}
    where $Ric_0=(n-1)I$ is the Ricci endomorphism of $\overline{g}_0=can$. Differentiating equation \eqref{eqricconform} at $t=0$,
    \begin{equation*}
    	\frac{\partial}{\partial t}\Big|_{t=0} Ric_t=-2(n-1)f I-(n-2) Hess\, f-\Delta fI.
   	\end{equation*}
    Therefore, by the assumption on $f$, $Ric_t(p)$ does not have an eigenvalue with multiplicity at least $n-1$, for all $t\neq 0$ sufficiently small. 
    
	Assume then, by contradiction, that some $t\neq 0$ sufficiently small is such that $(U,\overline{g}_t)$ has an isometric embedding into $\mathbb{R}^{n+1}$. Let $e_i$ be the eigenvectors of the shape operator $A_p$, and let $\lambda_i$ be the associated eigenvalue. Since $\overline{g}_t$ is conformally flat, its shape operator $A_p$ has an eigenvalue of multiplicity at least $n-1$. However, by Gauss equation, we have
	 \begin{equation*}
    	Ric_t(e_i)=(H-\lambda_i)\lambda_ie_i
    \end{equation*}
    where $H=\lambda_1+\ldots+\lambda_n$. It follows that $Ric_t$ also has an eigenvalue of multiplicity at least $n-1$, a contradiction.
\end{proof}

\subsection{Equivariance of the conformally flat deformations}\label{appequivarainceconformal}
In this section we sketch how the arguments of Section \ref{subsection equivariance} can be adapted in order to determine the isometry groups of the conformally flat metrics in $\mathcal{Z}$ constructed via Theorem A of \cite{ACM}, see Theorem \ref{symmetries of the conf-flat case} below. \\

	Given $f\in C^\infty_{odd}(\S^n)$, Theorem A of \cite{ACM} constructs a one-parameter family of conformally flat Riemannian spheres $(\S^n,e^{2\rho_t}\text{can})$ with $\rho_0=0$ and $\dot{\rho}_0=f$. As observed in Remark \ref{rmkconformalequivariance}, we can check that the construction is equivariant with respect to the natural action of $O(n+1)$, by following the arguments of Section \ref{subsectionapplication}. Thus, arguing as in Proposition \ref{propparathmB} and its Corollary \ref{corfirstpartthmB}, we conclude that the stabilizer group $G_f$ of $f$ is a subset of the group of isometries of $(\S^n,e^{2\rho_t}\text{can})$ (here, we identify elements of $O(n+1)$ with their restrictions to $\mathbb{S}^n$) for any $t$.
	
	Let $\R^{n+2}_1$ be the $(n+2)$-dimensional Lorentz space. Decompose it orthogonally as $\R^{n+2}_1=\R^{n+1}\times \R e$, where $e\in\R^{n+2}_1$ is a fixed vector with $\inner{e}{e}=-1$. 
Consider the maps
   \begin{align*}
        \iota_t : (\S^n,e^{2\rho_t}\text{can}) &\rightarrow \R^{n+2}_1\\
        x &\mapsto e^{\rho_t(x)}((x,0)+e).
    \end{align*}
	It is straightforward to check that $\iota_t$ is an isometric embedding, with the image contained in the upper light cone.

Let $O_1(n+2)$ be the orthogonal group of the Lorentz space $\R^{n+2}_1$. This group contains a copy of $O(n+1)$, acting naturally on the first factor of the decomposition $\R^{n+2}_1=\R^{n+1}\times\R e$ while fixing the vector $e$.

	Let $T$ be an isometry of $(\S^n,e^{2\rho_t}\text{can})$. Theorem 16.1 of \cite{BookDT} shows that, for $n\geq 3$, there exists an isometry $\tilde{T}\in O_1(n+2)$ such that $\tilde{T}|_{\iota_t(\S^n)}=T$ (\textit{cf}. Lemma \ref{lemintrinsicareextrinsic}). Hence, the isometry group $G_t$ of $(\mathbb{S}^n,e^{2\rho_t}\text{can})$ is identified with the set 
\begin{equation*}
	G_t=\{T\in O_1(n+2):T(\iota_t(\S^n))=\iota_t(\S^n)\}.
\end{equation*}
Observe that, given the natural inclusion $O(n+1)\subset O_1(n+2)$, for every $A\in G_f$ we have $A\iota_t=\iota_tA$. Thus, $G_f$ is a compact Lie subgroup of $G_t$.

	We want to show that $G_f=G_t$ for any sufficiently small $t\neq 0$, as soon as $f$ is $L^2$-orthogonal to the linear functions. This restriction is again natural, because we want to exclude trivial deformations corresponding to the action of the group of conformal diffeomorphism of $(\mathbb{S}^n,can)$. 
	
	We adapt the argument of Lemma \ref{lemma f=fA} to obtain:
\begin{lem}\label{f=fA conformal}
    Let $f\in C^\infty_{odd,0}(\S^n)$. Given $t_k\rightarrow 0$, let $T_k\in G_{t_k}\subseteq O_1(n+2)$. If $\lim T_k=T\in O_1(n+2)$ exists, then $T\in G_f$.
\end{lem}

	The main difference is that, in the last part of the argument, we use the scalar curvature $R_t$ of the metric $e^{2\rho_t}can$ instead of the mean curvature $H_t$ of the embedding $\iota_t$. Since $R_0=n(n-1)$ and
$$\partial_t|_0R_t(x)=-2n(n-1)(\Delta f(x)+nf(x)).$$
(see \cite{ACM}, Section 9.3), up to a constant, the properties of the Taylor expansion of $H_t$ required for the argument in Lemma \ref{lemma f=fA} are the same as those satisfied by $R_t$ (\textit{cf}. Remark \ref{rmkHeR}, where a similar observation played an analogous role).

	Finally, following the argument of Lemma \ref{lemg=gt} with obvious notational changes, we verify that $Lie(G_f)=Lie(G_t)$, as soon as $t$ is sufficiently small.
	
	Combining these three ingredients as in the proof of Theorem \eqref{thmB}, we prove:
\begin{thm}\label{symmetries of the conf-flat case}
    In the context of Theorem A of \cite{ACM} and its proof, suppose that $n\geq3$ and $f\in C^{\infty}_{odd,0}(\S^n)$. Then $G_t=G_{f}$ for every sufficiently small $t\neq 0$.
\end{thm}
\begin{proof}
	We may assume, by restricting the interval of definition of the deformation, that $\frac{1}{2}\leq e^{\rho_t(x)}\leq 2$ for any $t$ and $x$. 
	
	Suppose, by contradiction, that exists a sequence $t_k\rightarrow 0$ with $G_k:=G_{t_k}\neq G_f$. Since $G_f\subseteq G_k$, in this case there exist maps $T_k\in G_k\setminus G$. We claim that there exists a subsequence of $\{T_k\}$ that converges to some $T\in O_1(n+2)$. 
    
    In fact, let $v_1,\ldots, v_{n+2}\in\S^n$ be the vertices of a regular $(n+2)$-simplex, and consider $e_i:=v_i+e\in\R^{n+2}_1$.
    Observe that these vectors form a basis of the Lorentz space.
    Moreover, there exists a $w_i=w_{i,k}\in\S^n$ such that $T_k(\iota_{t_k}(v_i))=\iota_{t_k}(w_i)$, so
    $$|\inner{T_k(\iota_{t_k}(v_i)}{e_j}|=|\inner{\iota_{t_k}(w_i)}{e_j}|=e^{\rho_{t_k}(w_i)}|\inner{w_i+e}{v_i+e}|\leq2(1+1),$$
    thus
    $$|\inner{T_k(e_i)}{e_j}|=e^{-\rho_{t_k}(v_i)}|\inner{T_k(\iota_{t_k}(v_i)}{e_j}|\leq 8.$$
    Since $[-8,8]\subseteq\R$ is compact, we can assume, after taking a subsequence, that 
    $$\inner{Te_i}{e_j}:=\lim\inner{T_ke_i}{e_j},$$
    exists and is well-defined for all $i,j$. This defines a linear map $T=\lim T_k$ which belongs to $O_1(n+2)$ since this set is closed in the space of matrices. 
    
    By Lemma \ref{f=fA conformal}, we have $T\in G_f$. We follow now the steps of the proof by contradiction of Theorem \ref{thmB} to conclude. To be more precise, we replace $T_k$ by $T^{-1}T_k$ so $\lim T_k=Id$, decompose $Lie(O_1(n+2))=Lie(SO(n+1))\oplus Lie(SO(n+1))^\perp$, and use the exponential map of $O_1(n+2)$ as in the Euclidean case to find a non-zero $w\in Lie(G_f)\cap Lie(G_f)^\perp$, a contradiction.
\end{proof}

\bibliography{references}{}
\bibliographystyle{plain}

\end{document}